\documentclass[11pt, twoside, a4paper]{amsart}

\usepackage[utf8]{inputenc}
\usepackage{array,multirow}
\usepackage{amsfonts, amstext, amsmath, amsthm, amscd, amssymb}
\usepackage{mathrsfs}
\usepackage{tikz-cd}
\usetikzlibrary{cd}
\usepackage{url}
\usepackage[all]{xypic}
\usepackage[margin=1.3in]{geometry}

\usepackage{color}
\usepackage[colorlinks,
    linkcolor={red!50!black},
    citecolor={blue!50!black},
    urlcolor={blue!80!black}]{hyperref}

\usepackage{enumerate}

\renewcommand{\setminus}{{\smallsetminus}}

\newcommand{\bp}{\begin{pmatrix}}
\newcommand{\ep}{\end{pmatrix}}
\newcommand{\be}{\begin{equation}}
\newcommand{\ee}{\end{equation}}
\newcommand{\ol}[1]{\overline{#1}}

\newcommand{\id}{\operatorname{Id}}

\numberwithin{equation}{section}

\theoremstyle{plain}
\newtheorem{theorem}[equation]{Theorem}
\newtheorem{lemma}[equation]{Lemma}
\newtheorem{proposition}[equation]{Proposition}

\newtheorem{corollary}[equation]{Corollary}
\newtheorem*{question}{Question}

\newtheorem*{cor:m-doesnotdivide-n}{Corollary~\ref{cor:m-doesnotdivide-n}}
\newtheorem*{cor:examples}{Corollary~\ref{cor:examples}}
\newtheorem*{cor:smooth-vs-top}{Corollary~\ref{cor:smooth-vs-top}}
\newtheorem*{cor:infinitelynshakeslice}{Corollary~\ref{cor:infinitelynshakeslice}}
\newtheorem*{cor:infinitelymanyhomcobsurgeries}{Corollary~\ref{cor:infinitelymanyhomcobsurgeries}}
\newtheorem*{cor:1isweird}{Corollary~\ref{cor:1isweird}}
\newtheorem*{prop:torus}{Proposition~\ref{prop:torus}}
\newtheorem*{prop:shaking-num}{Proposition~\ref{prop:shaking-num}}

\newtheorem{thm}[equation]{Theorem}
\newtheorem{cor}[equation]{Corollary}
\newtheorem{lem}[equation]{Lemma}
\newtheorem{prop}[equation]{Proposition}

\theoremstyle{definition}
\newtheorem{example}[equation]{Example}

\newtheorem{definition}[equation]{Definition}

\newtheorem{remark}[equation]{Remark}

\theoremstyle{remark}

\numberwithin{equation}{section}

\def\Z{\mathbb Z}
\def\R{\mathbb R}
\def\Q{\mathbb Q}
\def\C{\mathbb C}

\def\wt#1{\widetilde{#1}}

\def\sm{\setminus}

\def\bp{\begin{pmatrix}}
\def\ep{\end{pmatrix}}
\def\ba{\begin{array}}
\def\ea{\end{array}}
\def\bn{\begin{enumerate}}
\def\en{\end{enumerate}}

\def\CP{\mathbb{CP}}
\def\RP{\mathbb{RP}}

\DeclareMathOperator\Arf{Arf}

\DeclareMathOperator\lk{lk}
\DeclareMathOperator\mult{mult}

\DeclareMathOperator\Wh{Wh}

\DeclareMathOperator\Hom{Hom}

\DeclareMathOperator\Id{Id}

\DeclareMathOperator\GL{GL}

\DeclareMathOperator\pt{pt}
\DeclareMathOperator\fr{fr}

\DeclareMathOperator\ks{ks}

\DeclareMathOperator\BO{BO}
\DeclareMathOperator\BTOP{BTOP}

\newcommand{\gs}{g^{\operatorname{top}}_4}
\newcommand{\galg}{g_{\operatorname{alg}}}
\newcommand{\gZ}{g_4^{\Z}}
\newcommand{\ts}{\textsuperscript}

\begin{document}

\title[Embedding spheres in knot traces]{Embedding spheres in knot traces}

\author[Feller]{Peter Feller}
\address{Department of Mathematics, ETH Z\"{u}rich, Switzerland}
\email{peter.feller@math.ethz.ch}

\author[Miller]{Allison N. Miller}
\address{Department of Mathematics, Rice University, Houston, Texas, USA}
\email{allison.miller@rice.edu}

\author[Nagel]{Matthias Nagel}
\address{Department of Mathematics, ETH Z\"{u}rich, Switzerland}
\email{matthias.nagel@math.ethz.ch}

\author[Orson]{\\ Patrick Orson}
\address{Department of Mathematics, Boston College, Chestnut Hill, USA}
\email{patrick.orson@bc.edu}

\author[Powell]{Mark Powell}
\address{Department of Mathematical Sciences, Durham University, United Kingdom}
\email{mark.a.powell@durham.ac.uk}

\author[Ray]{Arunima Ray}
\address{Max-Planck-Institut f\"{u}r Mathematik, Bonn, Germany}
\email{aruray@mpim-bonn.mpg.de }

\def\subjclassname{\textup{2020} Mathematics Subject Classification}
\expandafter\let\csname subjclassname@1991\endcsname=\subjclassname
%\expandafter\let\csname subjclassname@2000\endcsname=\subjclassname

\begin{abstract}
The trace of the $n$-framed surgery on a knot in $S^3$ is a 4-manifold homotopy equivalent to the 2-sphere. We characterise when a generator of the second homotopy group of such a manifold can be realised by a locally flat embedded $2$-sphere whose complement has abelian fundamental group. Our characterisation is in terms of classical and computable $3$-dimensional knot invariants. For each $n$, this provides conditions that imply a knot is topologically $n$-shake slice, directly analogous to the result of Freedman and Quinn that a knot with trivial Alexander polynomial is topologically slice.
\end{abstract}

\maketitle

\section{Introduction}

\begin{question}
  Let $M$ be a compact topological $4$-manifold and let $x \in \pi_2(M)$. Can $x$ be represented by a locally flat embedded $2$-sphere?
\end{question}

Versions of this fundamental question have been studied by many authors, such as~\cite{kervaire-milnor,  tristram, rochlin, hsiang-szczarba}. The seminal work of Freedman and Quinn~\cite{F, FQ} provided new tools with which to approach this problem. In independent work of Lee-Wilczy\'{n}ski~\cite[Theorem~1.1]{lee-wilczynski} and Hambleton-Kreck~\cite[Theorem~4.5]{hambleton-kreck:1993}, the methods of topological surgery theory were applied to provide a complete answer for simply connected, closed 4-manifolds, in the presence of a natural fundamental group restriction. That is, they classified when an element of the second homotopy group of such a 4-manifold can be represented by a locally flat embedded sphere whose complement has abelian fundamental group. Lee-Wilczy\'{n}ski~\cite{lee-wilczynski-AJM} later generalised their theorem to apply to simply connected, compact $4$-manifolds with homology sphere boundary.  In this article, we expand our understanding to another general class of 4-manifolds with boundary.

Our main result is an answer to the sphere embedding question for $4$-manifolds called \emph{knot traces}, with $x$ a generator of the second homotopy group.
Let $\nu K$ be an open tubular neighbourhood of a knot $K$ in $S^3$ and let $n$ be an integer. The \emph{$n$-framed knot trace} $X_n(K)$ is the smooth 4-manifold obtained by attaching a 2-handle $D^2\times D^2$ to the 4-ball along $\nu K\subset S^3$, using framing coefficient $n$ and smoothing corners.
The boundary of $X_n(K)$ is the \emph{$n$--framed surgery} $S_n^3(K) := (S^3 \sm \nu K) \cup D^2 \times S^1$, where $\partial D^2 \times \{1\}$ is attached to the $n$-framed longitude of $K$.

%We write $\Sigma_n(K)$ for the $n$-fold cyclic branched cover of $S^3$ with branching set $K$.
%Moved this definition below.

\begin{theorem}\label{theorem:Zn-shake-slice-thm}
 Let $K$ be a knot in $S^3$ and let $n$ be an integer.
A generator of~$\pi_2(X_n(K))$ can be represented by a locally flat embedded $2$-sphere whose complement has
abelian fundamental group if and only if:
  \begin{enumerate}[(i)]
    \item\label{item:trivialh1}
    $H_1(S_n^3(K);\Z[\Z/n])=0$; or  equivalently for $n \neq 0$,
    %  $ \displaystyle|H_1(\Sigma_{|n|}(K);\Z)| =
     $\prod_{\{\xi  \mid \xi^n=1\} }\Delta_K(\xi) =1$;
    \item\label{item:trivialarf} $\Arf(K)=0$; and
    \item\label{item:signaturesobstruct} $\sigma_K(\xi) =0$ for every $\xi \in S^1$ such that $\xi^n=1$.
     \end{enumerate}
  %Moreover, any two such spheres are ambiently  isotopic rel.\ boundary.
\end{theorem}

For certain choices of $n$, there are logical dependencies among the conditions \eqref{item:trivialh1}, (\ref{item:trivialarf}), and (\ref{item:signaturesobstruct}) above.  When $n=0$, condition \eqref{item:trivialh1} states that $H_1(S^3_0(K); \Z[\Z])=0$, which is equivalent to $\Delta_K(t)=1$, which in turn implies both conditions (\ref{item:trivialarf}) and (\ref{item:signaturesobstruct}). When $n=\pm1$, conditions (\ref{item:trivialh1}) and (\ref{item:signaturesobstruct}) are automatically satisfied.

When $n \neq 0$, condition~\eqref{item:trivialh1} is equivalent to the condition that $\Sigma_{|n|}(K)$, the $n$-fold cyclic branched cover of $S^3$ with branching set $K$, is an integral homology sphere. This is due to the classical formula $|H_1(\Sigma_{|n|}(K);\Z)| = \prod_{\{\xi  \mid \xi^n=1\} }\Delta_K(\xi)$ due to~\cite{Goeritz34, Fox-56-iii}. When $n \neq 0$ is even, (\ref{item:trivialh1}) implies (\ref{item:trivialarf}), as follows. The expression $\prod_{\{\xi  \mid \xi^n=1\} }\Delta_K(\xi)$ equals the \emph{resultant} $\operatorname{Res}(\Delta_K(t), t^n-1) \in \Z$. Whenever $m$ divides $n$, $t^m-1$ divides $t^n-1$, and so by the characterising properties of resultants we have
\[
\operatorname{Res}(\Delta_K(t), t^n-1) = \operatorname{Res}(\Delta_K(t), t^m-1) \cdot \operatorname{Res}(\Delta_K(t),(t^n-1)/(t^m-1)).
\]
Since  $\Arf(K)=1$ implies that $\operatorname{Res}(\Delta_K(t), t^2-1) = \Delta_K(-1) \neq 1$, it follows that $\Arf(K)=1$ further implies that $\operatorname{Res}(\Delta_K(t), t^n-1)= \prod_{\{\xi  \mid \xi^n=1\} }\Delta_K(\xi)$ is not equal to 1 for $n$ even.

Throughout the paper we will assume knots are oriented in order to make various constructions in the standard way. However, none of the conditions (\ref{item:trivialh1}), (\ref{item:trivialarf}), or (\ref{item:signaturesobstruct}) depends on a given orientation for the knot, and so the characterisation provided by Theorem~\ref{theorem:Zn-shake-slice-thm} is independent of this choice.

The remainder of the introduction proceeds as follows. In Section~\ref{sec:shakeintro}, we discuss the applications of Theorem~\ref{theorem:Zn-shake-slice-thm} to the study of whether a knot is \emph{shake slice}. In Section~\ref{sec:examples-simplest-cases} we give a quick proof of the main theorem for the case $n=0$, and then some further results that we have obtained when $n=\pm 1$. In Section~\ref{sec:outline} we outline the topological surgery theory strategy we use to prove our main result.

\subsection{Shake slice knots}\label{sec:shakeintro}
The embedding question for a generator of the second homotopy group of a knot trace is of interest via the lens of knot theory.

\begin{definition}
  A knot $K$ is \emph{$n$-shake slice} if a generator of $\pi_2(X_n(K)) \cong \Z$ can be realised by a locally flat embedded 2-sphere $S$.
 We say a knot $K$ is \emph{$\Z/n$-shake slice} if in addition $\pi_1(X_n(K) \sm S) \cong \Z/n$.
\end{definition}

For every $n \in \mathbb{Z}$, the fundamental group of the complement of an embedded sphere generating $\pi_2(X_n(K))$ abelianises to~$\Z/n$, so our condition that $\pi_1(X_n(K)) \cong \Z/n$ is just a more specific way to express the abelian condition on the fundamental group.

Theorem~\ref{theorem:Zn-shake-slice-thm} can be viewed as a characterisation of when a knot $K$ is $\Z/n$-shake slice.
Although the term `$n$-shake slice' was not coined until much later, classical obstructions to being $n$-shake slice were already obtained in the 1960s. In the course of proving Theorem~\ref{theorem:Zn-shake-slice-thm} we obtained several new proofs of these classical results. Robertello~\cite{robertello} showed that the Arf invariant obstructs $K$ from being $n$-shake slice, for every $n$. We give a new proof of Robertello's result in Proposition~\ref{prop:arf-obstruction-2} for all $n$, and we outline a second new proof in Remark~\ref{prop:Brown-Kervaire} for even $n$. Both proofs are different from Robertello's. Saeki~\cite{Saeki-92} has yet another  proof in the smooth category that uses the Casson invariant.

Tristram~\cite{tristram} showed that the signatures $\sigma_K(\xi_p^m)$, for $p$ a prime power dividing~$n$, obstruct $K$ being $n$-shake slice. We provide a different proof in Section~\ref{sec:obstructions} that also explains Tristram's theorem in the context of our results. Our proof is similar to that sketched by Saeki in~\cite[Theorem~3.4]{Saeki-92}.

For $n=0$, it remains unknown in both the smooth and topological categories whether every $0$-shake slice knot is slice.
An immediate result of our theorem is that a knot is $\Z$-shake slice if and only if it is $\Z$-slice: both correspond to Alexander polynomial one.

Now we describe some further consequences of Theorem~\ref{theorem:Zn-shake-slice-thm} that are proved as corollaries in Section~\ref{sec:shakeslice}. First, we provide new examples of the difference between the smooth and topological categories.  Recall that  knot is \emph{smoothly $n$-shake slice} if a generator of $\pi_2(X_n(K)) \cong \Z$ can be realised by a smoothly embedded 2-sphere $S$.

\begin{cor:smooth-vs-top}
For every $n>0$ there exist infinitely many knots that are $n$-shake slice but neither smoothly $n$-shake slice nor $($topologically$)$ slice. These knots may be chosen to be distinct in concordance.
%\fotnote{AM: I guess the strongest thing one might want here is not topologically concordant to a smoothly $n$-shake slice knot, but I don't know that we have this.}
\end{cor:smooth-vs-top}

We then show that being $n$-shake slice for infinitely many $n \in \mathbb{Z}$ does not imply slice.

\begin{cor:infinitelynshakeslice}
There exist infinitely many non-slice knots, each of which is $n$-shake slice for infinitely many $n \in \mathbb{Z}$. Moreover, these knots may be chosen to be distinct in concordance.
\end{cor:infinitelynshakeslice}

A question of Hedden~\cite{MPIM} asks whether the concordance class of a knot must be determined by the infinite tuple of homology cobordism classes $(S^3_{p/q}(K))_{p/q \in \mathbb{Q}}$.
We provide partial progress towards a negative answer to this question as follows, in what we believe is the first example of non-concordant knots with the property that infinitely many of their integer surgeries are homology cobordant.

\begin{cor:infinitelymanyhomcobsurgeries}
There exist infinitely many knots  $\{K_i\}$, mutually distinct in concordance, and an infinite family of integers $\{n_j\}$ such that $S^3_{n_j}(K_i)$ is homology cobordant to $S^3_{n_j}(U)$ for all $i, j\in \mathbb{Z}$.
%There exist infinitely many non-concordant knots $\{K_i\}$ and an infinite family of integers $\{n_j\}$ such that $S^3_{n_j}(K_i)$ is homology cobordant to $S^3_{n_j}(K_i')$ for all $i, i' \in \mathbb{Z}$.
\end{cor:infinitelymanyhomcobsurgeries}

We remark that $0$ is not an element of our family of integers $\{n_j\}$:  $S^3_0(K)$ is homology cobordant to $S^3_0(U)$ if and only if $K$ is slice in an integral homology 4-ball (which, incidentally, implies that $S^3_n(K)$ is homology cobordant to $S^3_n(U)$ for all $n \in \mathbb{Z}$.) It is currently an open question in both categories whether there exists a non-slice knot that is slice in a homology ball.

%As a result, there exist infinitely many knots  $\{K_i\}$, mutually distinct in concordance, and an infinite family of integers $\{n_j\}$ such that $S^3_{n_j}(K_i)$ is homology cobordant to $S^3_{n_j}(K_{i'})$ for all $i,i', j\in \mathbb{N}$.

We then show that, for most $m$ and $n$, the $m$-shake slice and $n$-shake slice conditions are independent.
%\begin{cor:m-doesnotdivide-n}
%  If $m$ does not divide $n$ then there exist infinitely many knots which are $n$-shake slice but not $m$-shake slice. These knots may be chosen to be distinct in concordance.
%\end{cor:m-doesnotdivide-n}

\begin{cor:examples}
  If $m$ does not divide $n$ then there exist infinitely many knots which are $n$-shake slice but not $m$-shake slice. These knots may be chosen to be distinct in concordance.
\end{cor:examples}

The difference between $n$-shake slice and $\Z/n$-shake slice can also be investigated using Theorem~\ref{theorem:Zn-shake-slice-thm}. For composite $n$ the theorem says there are extra signatures away from the prime power divisors of $n$ whose vanishing is necessary for an $n$-shake slice knot to be moreover $\Z/n$-shake slice; cf.~\cite[Proposition~3.7]{Saeki-92}. Consider that every slice knot is $n$-shake slice for all $n$, but according to Cha-Livingston~\cite{Cha-Livingston}, for any choice of composite $n$ and root of unity $\xi_n$, there exist slice knots $K$ with $\sigma_K(\xi_n)\neq 0$. We will prove that examples of this sort are not peculiar to composite $n$.

\begin{cor:1isweird}
For all $n \neq \pm1$, there exists a slice, and therefore $n$-shake slice, knot that is not $\Z/n$-shake slice.
\end{cor:1isweird}

\subsection{The cases  $n=0$ and $n=\pm1$}\label{sec:examples-simplest-cases}
The cases $n=0$ and $n=\pm1$ for Theorem~\ref{theorem:Zn-shake-slice-thm} are special, in that they can be proved relatively quickly by appealing directly to results of Freedman and Quinn. We provide a quick proof for $n=0$ now. Recall that when $n=0$ the three conditions in Theorem~\ref{theorem:Zn-shake-slice-thm} reduce to $\Delta_K(t)=1$.

\begin{example}[The case $n=0$]\label{ex:0}
First assume a generator of $\pi_2(X_0(K))$ is represented by an embedded sphere whose complement has fundamental group $\Z$. Since $n=0$, the normal bundle of this sphere is trivial. Perform surgery on $X_0(K)$ along this 2-sphere to obtain a pair $(V, S_0^3(K))$ where $\pi_2(V)=0$, and $\pi_1(V)\cong\Z$, generated by a meridian of $K$. Now attach a 2-handle to a meridian in the boundary. The cocore of the 2-handle is a slice disc $D'$ for $K$ in a homotopy 4-ball $B'$ such that $\pi_1(B' \smallsetminus D') \cong \Z$. Hence $\Delta_K(t)=1$.

Now assume $\Delta_K(t)=1$. By~\cite[Theorem~7]{Freedman:1984-1}\cite[11.7B]{FQ} (see also \cite[Appendix]{Garoufalidis-Teichner}), $K$ has a slice disc $D$ in $D^4$ with $\pi_1(D^4 \smallsetminus D) \cong \Z$. Cap this disc off with the core of the 2-handle to obtain the desired sphere in $X_0(K)$. This completes the proof of Theorem~\ref{theorem:Zn-shake-slice-thm} for $n=0$.
\end{example}

% To illustrate how the ``if'' direction becomes more complicated for $n\neq 0$, we sketch this for the $n=\pm1$ case (cf.~\cite[Proposition~3.6]{Saeki-92}). Recall that for $n=\pm 1$, the three conditions in Theorem~\ref{theorem:Zn-shake-slice-thm} reduce to $\Arf(K)=0$.

%\begin{remark}

To see a similarly quick proof for $n=\pm1$, the reader is invited to skip ahead to Example~\ref{ex:1}. In this case, we rely on the result of Freedman that as $S^3_{\pm1}(K)$ is an integer homology sphere, it must bound a contractible 4-manifold. When $n=\pm1$ we have obtained a wholly different proof of Theorem~\ref{theorem:Zn-shake-slice-thm} using Seifert surface constructions. These methods, detailed in Section~\ref{sec:n=1}, lead to the two results described next.

For each $n$, one can measure how far a knot $K$ is from being $n$-shake slice by considering the minimal genus of a locally flat embedded surface generating $H_2(X_n(K))$. This minimum is called the (topological) \emph{$n$-shake genus} $g_{\operatorname{sh}}^n(K)$. For $n=1$ we have a precise understanding of this invariant.

\begin{prop:torus}
For every knot $K$ there exists a locally flat embedded torus in $X_1(K)$ that generates $H_2(X_1(K))$ and has simply connected complement. In particular, \[g^1_{\operatorname{sh}}(K) = \Arf(K) \in \{0,1\}.\]
\end{prop:torus}

Now for a slice knot $K$, and for each $n$, a slice disc capped off by the core of the 2-handle in $X_n(K)$ geometrically intersects the cocore once. This suggests a way to measure of how far an $n$-shake slice knot $K$ is from being slice, by taking the minimum over all embedded spheres $S$ generating $\pi_2(X_n(K))$ of the geometric intersection number of $S$ with the cocore of the 2-handle of $X_n(K)$. We call this minimum the \emph{$n$-shaking number} of $K$.

For a knot $K$, $\gs(K)$ denotes the (topological) 4-genus or slice genus, the minimal genus among compact, oriented, locally flat surfaces in~$D^4$ with boundary $K$.  The (topological) $\Z$-slice genus~$\gZ(K)$ is the minimal genus among such surfaces whose complement has infinite cyclic fundamental group.
Computable upper bounds for~$\gZ(K)$ are discussed in~\cite{FellerLewark_16}, and include $2\gZ(K) \leq \deg(\Delta_K)$~\cite{MR3523068}.
For $n=1$, as well as proving again that a knot with vanishing $\Arf$ invariant is $1$-shake slice, the Seifert surface method provides explicit upper bounds on the $1$-shaking number.

\begin{prop:shaking-num}
For a knot $K$ with  $\Arf(K)=0$ we have
\[2\gs(K)+1\leq \text{$1$-shaking number of $K$}\leq 2\gZ(K)+1.
\]
In particular, for each integer $k \geq 0$ there exists a $1$-shake slice knot $K_k$ such that the $1$-shaking number of $K_k$ is exactly $2k+1$.
\end{prop:shaking-num}

Note that for every $n$, the $n$-shaking number is always odd since the algebraic intersection of the generator of $\pi_2(X_n(K))$ with the cocore of the 2-handle is~$1$. Thus this is a complete realisation result for $1$-shaking numbers.

\subsection{Proof outline}\label{sec:outline}
Having already proved Theorem~\ref{theorem:Zn-shake-slice-thm} for the case of $n=0$ above, we now restrict to $n\neq 0$, and will outline the proof of the `if' direction.
We note that the components of the `only if' direction  are proven respectively in Proposition~\ref{prop:trivialh1}, Proposition~\ref{prop:arf-obstruction-2}, and Proposition~\ref{prop:signaturesobstruct}.
%The uniqueness part of Theorem~\ref{theorem:Zn-shake-slice-thm} is proved in Section~\ref{section:uniqueness} using the Perron-Quinn pseudoisotopy theorem.
% and then give the proofs in the cases $n=0$ and $n=\pm1$ as an illustration of the method.

Given a group $\pi$ and closed $3$-manifolds $M_1$ and $M_2$, together with homomorphisms $\varphi_i\colon\pi_1(M_i)\to\pi$, we say a cobordism $W$ from $M_1$ to $M_2$ is \emph{over} $\pi$ if $W$ is equipped with a map $\pi_1(W)\to \pi$ restricting to the given homomorphisms on the boundary.

The key idea of our proof is that
for a fixed knot $K$, a generator of $\pi_2(X_n(K))$ can be represented by a locally flat embedded 2-sphere $S$ with $\pi_1(X_n(K) \sm S) \cong \Z/n$
if and only if there exists a homology cobordism $V$ from $S^3_n(K)$ to the lens space $L(n,1)$ over $\Z/n$, extending standard maps $\pi_1(S^3_n(K))\to \Z/n$ and $\pi_1(L(n,1))\to \Z/n$, such that  $\pi_1(V) \cong \Z/n$ and $V \cup_{L(n,1)} D_n$ is homeomorphic to $X_n(K)$, where $D_n$ denotes the $D^2$-bundle over $S^2$ with euler number $n$.

 The proof of the ``if'' direction involves constructing such a cobordism $V$ when the list of invariants in Theorem~\ref{theorem:Zn-shake-slice-thm} vanish. Here is an outline.

\begin{enumerate}
\item Show there exists a cobordism $W$ between $S^3_n(K)$ and the lens space $L(n,1)$  and a map $W \to L(n,1)\times [0,1]$ that restricts to a degree one normal map (Definition \ref{defn:DONM}) $f\colon S^3_n(K) \to L(n,1) \times \{0\}$ and the identity map $L(n,1) \to L(n,1) \times \{1\}$ (Section~\ref{sec:surgeryproblem}). This uses the assumption that $\Arf(K)=0$ when $n$ is even and no assumptions when $n$ is odd.
\item Use the computation of the simple surgery obstruction groups $L_4^s(\Z[\Z/n])$ in terms of multisignatures
 to show that we can replace $W$ with a homology cobordism $V$ between $S^3_n(K)$ and $L(n,1)$ over $\Z/n$, with $V$ homotopy equivalent to $L(n,1) \times I$ (Section~\ref{sec:homologycob}). This uses the assumptions that $H_1(S_n^3(K);\Z[\Z/n])=0$ and $\sigma_K(\xi_n^m) =0$ for all $m$.
\item Let $X:= V \cup_{L(n,1)} D_n$. Note that a generator for $\pi_2(X)$ is represented by an embedded sphere in $D_n$.
Use Boyer's classification (Theorem \ref{thm:Boyer}) to conclude that $X$ is homeomorphic to $X_n(K)$ (Section~\ref{sec:fake-trace}). This uses the assumption that $\Arf(K)=0$ when $n$ is odd and no additional assumptions when $n$ is even. More precisely, according to the classification, $X$ is homeomorphic to $X_n(K)$ automatically when $n$ is even, and if and only if $\ks(X)=\ks(X_n(K))=0$ when $n$ is odd; the latter equality follows since $X_n(K)$ is smooth. We show in Proposition~\ref{proposition:identifying} that $\Arf(K)=\ks(X)$.
\end{enumerate}

An interesting aspect of the proof (of both the ``if'' and ``only if'' directions) of Theorem~\ref{theorem:Zn-shake-slice-thm} is that the Arf invariant appears in different places for $n$ odd and $n$ even. However, in each case its vanishing is required.

 Recent work of Kim and Ruberman~\cite{Kim-Ruberman-2019} uses similar techniques to prove the existence of topological spines in certain 4-manifolds. Their argument is in some ways structurally quite similar to ours, since both works follow a surgery theoretic strategy.
%, relying on constructing a normal cobordism $W$ between two 3-manifolds, analyzing the surgery obstruction to improving $W$ to be a homology cobordism with simple fundamental group, showing that one can make this obstruction vanish, and then applying Boyer's work classifying simply connected 4-manifolds up to homeomorphism.\fotnote{AM: This sentence is far too long.}
There is no overlap between our results: every knot trace $X_n(K)$ admits a PL-spine consisting of the cone on $K$ union the core of the attached 2-handle, as is crucially used in \cite{Kim-Ruberman-2019}. Moreover, there is a key difference: Kim and Ruberman have flexibility in their choice of a second 3-manifold, whereas we have a fixed choice of $S^3_n(K)$ and $L(n,1)$.

\subsection*{Conventions}
From Section \ref{sec:obstructions} onwards, we assume for convenience that $n>0$. The case of $n=0$ was proved in Example \ref{ex:0}. When $n<0$, the argument is the same as for $-n$.
Throughout, manifolds are compact and oriented, and knots are oriented.

\subsection*{Acknowledgments}

We are grateful to Peter Teichner for explaining a proof of the $n=\pm 1$ case to us in 2016, and to Danny Ruberman and Fico Gonz\'{a}lez-Acu\~{n}a for interesting discussions.
This project started during the ``Workshop on 4-manifolds'' at the Max Planck Institute for Mathematics in Bonn in the autumn of 2019, and we thank the organisers of this workshop and the MPIM.
ANM is supported by NSF DMS-1902880.
PF and MN gratefully acknowledge support by the SNSF Grant~181199.

%\subsection*{Outline of the paper}
%In Section~\ref{sec:obstructions}, we establish the necessity of conditions (1) and (3) in Theorem~\ref{theorem:Zn-shake-slice-thm}. Sections~\ref{sec:surgeryproblem},~\ref{sec:homologycob}, and~\ref{sec:fake-trace} prove the sufficiency of the three conditions.

\section{Corollaries to Theorem~\ref{theorem:Zn-shake-slice-thm}}\label{sec:shakeslice}

Before embarking on the main work of proving Theorem~\ref{theorem:Zn-shake-slice-thm}, we use it to prove several knot theoretic corollaries.

\begin{cor}\label{cor:cables}
Let $K$ be any knot  and let $n$ be an integer. Suppose that $n$ is even or $\Arf(K)=0$. Then $C_{n,1}(K)$ is $\Z/n$-shake slice.
\end{cor}

\begin{proof}
We use the formulae for  the Alexander polynomial and signatures of a satellite knot, due respectively to Seifert~\cite{Seifert:1950-1} and Litherland~\cite{Litherland:1979-1}, to verify the conditions of Theorem~\ref{theorem:Zn-shake-slice-thm} for $C_{n,1}(K)$.

 Since $C_{n,1}(U)=U$ we have
$\Delta_{C_{n,1}(K)}(t)= \Delta_K(t^n)$ and $\sigma_{C_{n,1}(K)}(\omega)= \sigma_K(\omega^n)$ for all $ \omega \in S^1$.
Letting $\omega_n$ denote a primitive $n$th root of unity, we therefore have for $1 \leq k \leq n$ that
\[ \sigma_{C_{n,1}(K)}(\omega_n^k)= \sigma_K(\omega_n^{nk})= \sigma_K(1) =0.\]
%Drawing on the formula for the order of first homology of cyclic branched covers~\cite{Goeritz34, Fox-56-iii},
We have that
\begin{align*}
%|H_1\left(\Sigma_n(C_{n,1}(K))\right)| &=
\prod_{k=1}^n \Delta_{C_{n,1}(K)}(\omega_n)= \prod_{k=1}^n \Delta_{K}(\omega_n^n) =1.
\end{align*}
%It is shown by Levine in~\cite[Proposition~3.4]{Levine:1966-1}
Levine~\cite[Proposition~3.4]{Levine:1966-1} showed that $\Arf(J)=0$ if and only if $\Delta_J(-1)\equiv \pm 1 \pmod 8$, and so since
\[\Delta_{C_{n,1}(K)}(-1)= \Delta_{K}((-1)^n)=  \begin{cases}
\Delta_K(-1) & \text{ $n$ odd} \\ 1 & \text{ $n$ even.} \end{cases}\]
we obtain as desired that $\Arf(C_{n,1}(K))=0$.
\end{proof}

\begin{remark}\label{rem:easyway}
Gordon~\cite{Gordon-83-Dehn} observed that for any knot $K$
\[ S^3_n(C_{n,1}(K))= L(n,1)\, \#\, S^3_{1/n}(K).\]
This gives a slightly more direct argument that  $C_{n,1}(K)$ is $\Z/n$-shake slice whenever $n$ is even or $\Arf(K)=0$, as follows. Let $C$ be the contractible 4-manifold with boundary  $S^3_{1/n}(K)$ guaranteed by~\cite[Theorem~1.4$'$]{F}, and define $V$ to be the boundary connected sum of $L(n,1) \times I$ and $C$. The manifold $V$ is now a homology cobordism from $S^3_n(C_{n,1}(K))$ to $L(n,1)$ that is homotopy equivalent to $L(n,1) \times I$. This allows one to skip the work of Sections~\ref{sec:surgeryproblem} and~\ref{sec:homologycob} constructing a homology cobordism and go straight to proving that $V \cup D_n$ is homeomorphic to $X_n(C_{n,1}(K))$ using the results we prove in Section \ref{sec:fake-trace} and Boyer's classification (Theorem \ref{thm:Boyer}), which when $n$ is odd requires $\Arf(C_{n,1}(K))= \Arf(K)=0$.
\end{remark}

We can use Corollary~\ref{cor:cables} to obtain many new examples of $n$-shake slice knots.

\begin{corollary}\label{cor:smooth-vs-top}
For every $n \neq 0$ there exist infinitely many topological concordance classes of  knots that are $n$-shake slice but not smoothly $n$-shake slice.
\end{corollary}

\begin{proof}
In both the smooth and topological categories a knot $K$ is $n$-shake slice if and only if $-K$ is $(-n)$-shake slice. So it suffices to show the $n>0$ case as follows.
% Also we show that our examples are $n$-shake slice by showing they are $\Z/n$-shake slice, and a corollary to Theorem~\ref{theorem:Zn-shake-slice-thm} (Corollary~\ref{cor:n-and-minus-n}) is that a knot $K$ is $\Z/n$-shake slice if and only if it is $\Z/(-n)$-shake slice.  So for $n<0$, apply the argument below with $-n$ in place of $n$, and the knots $-C_{-n,1}(K_j)$.

We can use the description of $S^3_n(C_{n,1}(K))$ from Remark~\ref{rem:easyway} to show that $C_{n,1}(K)$ is often not smoothly $n$-shake slice. Given an integer homology sphere $Y$, Ozsv{\'a}th-Szab{\'o} associate a so-called $d$-invariant $d(Y) \in \mathbb{Q}$, with the property that $d(Y)=0$ if $Y$ bounds a rational homology ball~\cite{Ozsvath-Szabo:2003-2}.  We show that if $C_{n,1}(K)$ is smoothly $n$-shake slice for some $n>0$, then $d(S^3_{1}(K))=0$ as follows.

Suppose that $C_{n,1}(K)$ is smoothly $n$-shake slice via a sphere $S$ in $X_n(J)$. The  exterior of $S$ can be quickly confirmed to be a smooth homology cobordism between $S^3_n(J)$ and $L(n,1)$; see Lemma~\ref{lem:shake-slice-to-hom-cob}.
%
%To obstruct the existence of such a homology cobordism we use the Heegaard-Floer correction term $d(Y, \mathfrak{s})$} of a 3-manifold $Y$ and a Spin$^c$-structure~$\mathfrak{s}$.
%
%The homology cobordism invariance of $d$ invariants implies that
%\begin{equation}\label{eqn:d-inv} \{d(S^3_n(J),\mathfrak{s}_i)\}_{i=0}^{n-1}= \{d(L(n,1),\mathfrak{t}_j)\}_{j=0}^{n-1},\end{equation}
%where $\{\mathfrak{s}_i\}_{i=0}^{n-1}$ and $\{\mathfrak{t}_j\}_{j=0}^{n-1}$ are arbitrary enumerations of the Spin$^c$ structures on $S^3_n(J)$ and $L(n,1)$, respectively.
%for a 3-manifold $Y$ and a Spin$^c$-structure $\mathfrak{s}$, we use $d(Y, \mathfrak{s}) \in \mathbb{Q}$ to refer to the Heegaard-Floer correction
%term of Ozsv{\'a}th-Szab{\'o}~\cite{Ozsvath-Szabo:2003-2} and $\{\mathfrak{s}_i\}_{i=0}^{n-1}$ and $\{\mathfrak{t}_j\}_{j=0}^{n-1}$ are arbitrary enumerations of the Spin$^c$ structures on $S^3_n(J)$ and $L(n,1)$, respectively.
Therefore $S^3_n(C_{n,1}(K))= L(n,1)\, \#\, S^3_{1/n}(K)$  and  $S^3_n(U)= L(n,1)$  are homology cobordant via some  smooth $W$.  By summing $W$ with $-L(n,1) \times I$ along $D^4 \times I \subset W$, we further obtain that $S^3_{1/n}(K)$ is smoothly rationally homology cobordant to $S^3$. Therefore, $d(S^3_{1/n}(K))= d(S^3)=0$. 
Furthermore, since $n >0$ we have $d(S^3_{1/n}(K))= d(S^3_1(K))$ by~\cite[Proposition 1.6]{Ni-Wu-2015}.

Now for each $j \in \mathbb{N}$ let $K_j:= T_{2,8j+1}$.  Note that since $K_j$ is alternating and the ordinary signature $\sigma_{K_j}(-1) <0$, \cite[Corollary 1.5]{Ozsvath-Szabo-03-alt} implies that $d(S^3_1(K)) \neq 0$. So $C_{n,1}(K_j)$ is not smoothly $n$-shake slice, despite being $\Z/n$-shake slice by Corollary~\ref{cor:cables}. One can use Litherland's satellite formula~\cite{Litherland:1979-1} to compute that the first jump of the Tristram-Levine signature function of $C_{n,1}(K_j)$ occurs  at $e^{2 \pi i \theta_j}$, where $\theta_j=\frac{1}{2n(8j+1)}$. Therefore the knots $C_{n,1}(K_j)$ are distinct in concordance.
%
%Finally, as noted in Corollary~\ref{cor:examples}, we can ensure that for some $m \in \mathbb{Z}$ we have that $C_{n,1}(K)$ is not topologically $m$-shake slice, and hence is certainly not topologically slice.
\end{proof}

In another direction, we are able to show that there exist non-slice knots which are nevertheless $n$-shake slice for infinitely many $n \in \mathbb{Z}$. It is presently unknown in either category whether $0$-shake slice implies slice, and so the question of whether being $n$-shake slice for \emph{all} $n \in \mathbb{Z}$ implies slice appears both interesting and difficult.

\begin{cor}\label{cor:infinitelynshakeslice}
There exist knots that are $\Z/n$-shake slice for every prime power $n \in \mathbb{Z}$, but are not slice. Moreover, these knots may be chosen to represent infinitely many concordance classes.
\end{cor}

\begin{proof}
Let $J$ be a knot with Alexander polynomial equal to the $m^{th}$ cyclotomic polynomial, where $m$ is divisible by at least 3 distinct primes.
Since $J \#-J$ does not have trivial Alexander polynomial,
there exist infinitely many non-concordant knots sharing its Seifert form~\cite{Kim05}.
We show that any such knot $K$ is $\Z/n$-shake slice for all prime powers $n$.

By~\cite{Livingston02Seifert}, we have $|H_1(\Sigma_{|n|}(J))|=1$, so condition $(i)$ of Theorem~\ref{theorem:Zn-shake-slice-thm} follows immediately:
\[\prod_{\{\xi  \mid \xi^n=1\} }\Delta_K(\xi) =
 |H_1(\Sigma_{|n|}(K))|= |H_1(\Sigma_{|n|}(J))|^2=1.\]
For conditions $(ii)$ and $(iii)$, observe that
 since $K$ shares a Seifert form with the knot $J \#-J$, we have that $\Arf(K)=0$ and for every $\omega \in S^1$ we have
 \[\sigma_{\omega}(K)=\sigma_{\omega}(J)+ \sigma_{\omega}(-J)=0. \qedhere\]
\end{proof}

We remark that it is open whether there is a smoothly non-slice knot that is smoothly $n$-shake slice for infinitely many $n$. The smooth analogue of the next result is also open.

\begin{cor}\label{cor:infinitelymanyhomcobsurgeries}
There exist infinitely many knots $\{K_i\}$, mutually distinct in concordance, and an infinite family of integers $\{n_j\}$ such that $S^3_{n_j}(K_i)$ is homology cobordant to $S^3_{n_j}(U)$ for all $i, j\in \mathbb{N}$. %Moreover, these knots may be chosen to represent infinitely many concordance classes of knots.
\end{cor}
\begin{proof}
This follows immediately from Corollary~\ref{cor:infinitelynshakeslice} and Lemma~\ref{lem:shake-slice-to-hom-cob}.
\end{proof}

We are also able to show that homology cobordism of the $n$-surgeries is often not enough to determine that a knot is $n$-shake slice, as follows.

\begin{cor}\label{cor:homcobnotshakeslice}
For each odd $n \in \mathbb{N}$,  there exists
 $K$ such that $S^3_n(K)$ and $S^3_n(U)$ are topologically homology cobordant but $K$ is not $n$-shake slice.
  In fact, for each odd $n$ there exist knots representing infinitely many concordance classes that satisfy this.
\end{cor}
\begin{proof}
By the proof of Theorem~\ref{theorem:Zn-shake-slice-thm}, if $n$ is odd and  $K$ is a knot with
$\prod_{\{\xi  \mid \xi^n=1\} }\Delta_K(\xi) =1$ and $\sigma_K(\xi)=0$ for all $n$th roots of unity $\xi$, then $S^3_n(K)$ and $S^3_n(U)= L(n,1)$ are homology cobordant. As discussed in the proof of Corollary~\ref{cor:smooth-vs-top}, for every knot $J$ the knot $K= C_{n,1}(J)$ satisfies these conditions. However, if $\Arf(J) \neq 0$, then since $\Arf(K)= \Arf(J)$ we obtain that $C_{n,1}(J)$ is not $n$-shake slice (or even, for the cognoscenti, $n$-shake concordant to the unknot).
Therefore the set $\{C_{n,1}(T_{2, 8j+3}\}_{j\geq 1}$ is an infinite collection of such knots, all distinguished in concordance by the first jump of their Tristram-Levine signature functions, which occur at $e^{2 \pi i y_j}$, where $y_j= \frac{1}{2n(8j +3)}$.
\end{proof}

We remark that in the examples of Corollary~\ref{cor:homcobnotshakeslice},  $S^3_n(K)$ and $S^3_n(U)$ are homology cobordant  not just with integer coefficients,  but also with $\Z[\Z/n]$-coefficients.

We also compare the difference between $m$-shake slice and $n$-shake slice for $m\neq n$.

\begin{corollary}\label{cor:m-doesnotdivide-n}
  If $m \mid n$ and $K$ is $\Z/n$-shake slice, then $K$ is $\Z/m$-shake slice.
\end{corollary}

\begin{proof}
First, note that if $n=0$ and $K$ is $\Z/n$-shake slice then $\Delta_K(t)=1$, and so the conditions $(i)$, $(ii)$, and $(iii)$ of Theorem~\ref{theorem:Zn-shake-slice-thm} are satisfied for all $m \in \Z$.
%by work of Freedman there exists a slice disc $\Delta$ for $K$ such that $\pi_1(B^4 \smallsetminus \nu(\Delta)) \cong \Z$. It follows that $K$ is $\Z/m$-shake slice for all $m \in \Z$.

So assume $n \neq 0$.
Since $\prod_{\{\xi  \mid \xi^m=1\} }\Delta_K(\xi)$ divides $\prod_{\{\xi  \mid \xi^n=1\} }\Delta_K(\xi)$ and both are integers, 
if the criterion~\eqref{item:trivialh1} holds for~$n$, then it also holds for~$m$. 
The signature and Arf invariant conditions are immediate.
\end{proof}

\begin{corollary}\label{cor:examples}
  If $m$ does not divide $n$ then there exist infinitely many knots that are $n$-shake slice but not $m$-shake slice. These knots may be chosen to be distinct in concordance.
\end{corollary}

\begin{proof}
Let $q$ be a prime power which divides $m$ but not $n$.
Let $K$ be any knot with $\Arf(K)=0$ and  $\sigma_{K}(e^{2 \pi ik/q}) \neq 0$ for all $k=1, \dots, q-1$. Such knots are easy to find, for example by taking $K=T_{2, 8N+1}$ for sufficiently large $N$.
 Now let $J=C_{n,1}(K)$, and note that since $\Arf(K)=0$, Corollary~\ref{cor:cables} tells us that $J$ is $n$-shake slice.
However,
$\sigma_{C_{n,1}(K)}(e^{2 \pi i/q})=  \sigma_{K}(e^{2 \pi in/q}) \neq 0$,
since $q$ does not divide~$n$. So $J$ is not $m$-shake slice.

We can obtain $\{J_j\}_{j \in \mathbb{N}}$ representing infinitely many concordance classes of $n$-shake but not $m$-shake slice knots  by letting $J_j=C_{n,1}(T_{2,8(N+j)+1})$ for sufficiently large $N$.  By choosing knots with distinct signature functions, such as the family $T_{2,8(N+j)+1}$, again for large $N$.
  As in the proof of Corollary~\ref{cor:smooth-vs-top}, these knots are distinguished in concordance by the first jump of the Tristram-Levine signature, which occurs for $J_j$ at $e^{2 \pi i x_j}$, where $x_j= \frac{1}{2n(8(N+j)+1)}$.
\end{proof}

\begin{remark}
For many pairs $(n,m)$ such that $m$ does not divide $n$, one can also find examples of $n$-shake slice knots which are not $m$-shake slice by considering certain linear combinations of $(2,2k+1)$ torus knots. For example, one can verify that $K_n:=6T_{2,4n+1} \# - 4T_{2, 6n+1}$ satisfies the conditions of Theorem~\ref{theorem:Zn-shake-slice-thm} for $n$, and hence is $\Z/n$-shake slice. Additionally, computation of Tristram-Levine signatures shows that if $m$ does not divide $2n$, then  $K_n$ is not $m$-shake slice.
\end{remark}

Finally, we show  in almost all cases that  $\Z/n$-shake slice is a strictly stronger condition than  $n$-shake slice.

\begin{corollary}\label{cor:1isweird}
A knot is $(\pm 1)$-shake slice if and only if it is $\Z/1$-shake slice.
For all other $n$ there exists a slice, and therefore $n$-shake slice, knot that is not $\Z/n$-shake slice.
\end{corollary}

\begin{proof}
One direction of the first sentence is obvious. For the other direction,
observe that an embedded sphere $S$ representing a generator of $\pi_2(X_{\pm 1}(K))$ has a normal bundle with euler number $\pm 1$ (see Lemma~\ref{lem:lambda-mu-e}). In other words, the sphere $S$ has a push-off intersecting it precisely once. Thus the meridian of $S$ is null-homotopic in the complement of $S$. Since the meridian of $S$ normally generates the fundamental group of the complement of $S$, the statement follows.

The knot $K=4_1 \# 4_1$ is slice and hence $n$-shake slice for all $n$. But $K$ has $|H_1(\Sigma_n(K))| \neq 1$ for all $n>1$, and hence is not $\Z/n$-shake slice for $|n|>1$ by condition (\ref{item:signaturesobstruct}) of Theorem~\ref{theorem:Zn-shake-slice-thm}. Recall that when $n=0$, the conditions of Theorem~\ref{theorem:Zn-shake-slice-thm} reduce to triviality of the Alexander polynomial. But as $\Delta_K(t)=(t^2-3t+1)^2$, this knot $K$ is not $\Z/0$-shake slice either.
\end{proof}

When $n=\pm1$, the conditions (\ref{item:trivialh1}) and (\ref{item:signaturesobstruct}) of Theorem~\ref{theorem:Zn-shake-slice-thm} are automatically satisfied, so Corollary~\ref{cor:1isweird}  immediately gives the following.

\begin{corollary}\label{cor:1iffArf}
A knot $K$ is $(\pm 1)$-shake slice if and only  if  it is $\Z/ (\pm1)$-shake slice if and only if $\Arf(K)=0$.
\end{corollary}

If a knot is $n$-shake slice then it has vanishing Arf invariant~\cite{robertello}, so Corollary~\ref{cor:1iffArf} immediately gives the following.

\begin{corollary}
If a knot $K$ is $n$-shake slice for some integer $n$, then it is $(\pm 1)$-shake slice.
\end{corollary}

Changing the orientation on an $n$-trace~$X_n(K)$ results in the trace~$X_{-n}(-K)$, and so $K$ is $n$-shake slice if and only if $-K$ is $(-n)$-shake slice.
Surprisingly, the conditions in Theorem~\ref{theorem:Zn-shake-slice-thm} show that an even stronger symmetry holds, as in the following corollary.
%Although we cannot relate the $4$--manifolds~$X_n(K)$ and $X_{-n}(K)$ directly,
%the conditions in Theorem~\ref{theorem:Zn-shake-slice-thm} do not depend on the sign of $n$, and the next corollary is immediate.

\begin{corollary}\label{cor:n-and-minus-n}
  A knot $K$ is $\Z/n$-shake slice if and only if it is $\Z/(-n)$-shake slice.
\end{corollary}

Here of course $\Z/n \cong \Z/(-n)$, but the 4-manifolds~$X_n(K)$ and $X_{-n}(K)$ are generally different. In particular, we know of no reason to believe that a knot is $n$-shake slice if and only if it is $(-n)$-shake slice.

For any $n$, it remains unknown in both the smooth and topological category whether $n$-shake slice knots have $n$-shake slice connected sum. As the invariants involved in conditions \eqref{item:trivialh1}, (\ref{item:trivialarf}), and (\ref{item:signaturesobstruct}) of Theorem~\ref{theorem:Zn-shake-slice-thm} are all additive under connected sum we have the following immediate corollary.

\begin{corollary}\label{cor:connected-sum}
If $K$ and $J$ are $\Z/n$-shake slice, then so is $K \#J$.
\end{corollary}

%\begin{remark}
%Our argument showing that the knots $\{J_j\}_{j \in \mathbb{N}}$ are distinct in concordance in fact shows that they generate a $\Z^{\infty}$ subgroup of the topological concordance group. Moreover, since the conditions of Theorem~\ref{theorem:Zn-shake-slice-thm} behave well under connected sum, all elements of this subgroup are $\Z/n$-shake slice and hence $n$-shake slice. However, it is not clear that only the trivial element of this subgroup is $m$-shake slice.
%\end{remark}

\section{Obstructions to $\Z/n$-shake sliceness}\label{sec:obstructions}

In this section we prove that the conditions listed in Theorem~\ref{theorem:Zn-shake-slice-thm} on the knot signatures and the Alexander polynomial
%order of the homology of the cyclic branched cover
%\fotnote{AR: flagging this as instance of where we referred to the old formulation}
 are indeed necessary conditions for a knot to be $\Z/n$-shake slice.
Some of these results have been shown by Tristram~\cite{tristram} and Saeki~\cite{Saeki-92}; we include our own proofs for completeness, for the convenience of the reader, and to introduce the identification of certain Atiyah-Singer/Casson-Gordon signatures with Tristram-Levine knot signatures that will be needed later on.
As stated in the conventions section, we henceforth assume that $n>0$.

\begin{lemma}\label{lem:shake-slice-to-hom-cob}
If a knot $K$ is $n$-shake slice via an embedded sphere $S$ in $X_n(K)$ then $W:= X_n(K) \smallsetminus \nu(S)$ is a homology cobordism from $S^3_n(K)$ to $L(n,1)$. If $K$ is further $\Z/n$-shake slice via $S$, then $\pi_1(W) \cong \Z/n$.
\end{lemma}
\begin{proof}
The first statement follows from computation using the Mayer-Vietoris sequence for $X_n(K)= W \cup \nu(S)$.  The second statement is immediate from the definition.
\end{proof}

%Suppose that $K$ is $n$-shake slice via an embedded sphere $S$ in $X_n(K)$.
%Then
%$W:= X_n(K) \smallsetminus \nu(S)$ is a homology cobordism from $S^3_n(K)$ to $L(n,1)$, as can be seen by contemplation of the Mayer-Vietoris sequence for $X_n(K)= W \cup \nu(S)$.  If we moreover assume that $K$ is $\Z/n$-shake slice via $S$, then by definition we have that $\pi_1(W) \cong \Z/n$.

\begin{remark}\label{rem:HCobShakeToLens}
In this section, we only use the hypothesis that the manifolds $S^3_n(K)$ and $L(n,1)$ are homology cobordant via a homology cobordism $W$ with $\pi_1(W) \cong \Z/n$. This hypothesis is sufficient to establish conditions \eqref{item:trivialh1} and \eqref{item:signaturesobstruct} of Theorem~\ref{theorem:Zn-shake-slice-thm}
%While reading the subsequent arguments, the careful reader will observe that if our objective were only to prove that (\ref{item:trivialh1}) and (\ref{item:signaturesobstruct}) of Theorem~\ref{theorem:Zn-shake-slice-thm} are satisfied, it would be enough to make the weaker assumption that there exists a homology cobordism $W$ from $S^3_n(K)$ to $L(n,1)$ with $\pi_1(W)\cong \Z/n$.
However, one cannot make this weaker assumption, at least when $n$ is odd, if we wish to conclude that $\Arf(K)=0$. In particular, for every knot $K$  we know that $S^3_1(K)$ is homology cobordant to $L(1,1)=S^3$ by a simply connected cobordism ~\cite[Theorem~1.4$'$]{F} (see also~\cite[9.3C]{FQ}), but of course some knots have Arf invariant 1.  On the other hand, the existence of a \emph{smooth} homology cobordism would be enough to imply that $\Arf(K)=0$~\cite{Saeki-92}.
\end{remark}

\begin{prop}\label{prop:trivialh1}
Let $K$ be a knot such that $S^3_n(K)$ and $L(n,1)$ are homology cobordant via a homology cobordism $W$ with $\pi_1(W) \cong \Z/n$.
%If $K$ is $\Z/n$-shake slice
Then $H_1(S^3_n(K); \Z[\Z/n])=0= H_1(\Sigma_n(K);\Z)$.
\end{prop}
\begin{proof}
%Let $S$ be the hypothesised embedded sphere in $X_n(K)$ with exterior $W$. As noted above, we have that  $\pi_1(W)\cong \Z/n$ and that $W$ is a homology cobordism from $S^3_n(K)$ to $L(n,1)$.
We will show that the $n$-sheeted cyclic cover $Y_n:=(S^3_n(K))_n$  of $S^3_n(K)$ has trivial integral homology. Since $H_1(Y_n;\Z)=0$ implies that $H_1(S^3_n(K); \Z[\Z/n])=0$, this will imply the first part of our desired result.
By the universal coefficient theorem, it suffices to show that $H_1(Y_n ;\mathbb{F}_p)=0$ for all primes $p$.

Let $p$ be a prime. For each $k \in \mathbb{N}$ and space $X$, let $b_k^p(X)$ denote the dimension of $H_k(X; \mathbb{F}_p)$ as a $\mathbb{F}_p$-vector space.
 Let $\wt{W}$ be the $\Z/n$-cover of $W$, and observe that since $\wt{W}$ is simply connected we have $H_1(\wt{W}; \mathbb{F}_p)=0$. It follows that
 \[ 0=n \cdot \chi(W)= \chi(\wt{W})= 1+ b_2^{p}(\wt{W})- b_3^{p}(\wt{W}).\]
Here the first equality holds since $W$ has the same Euler characteristic as a closed 3-manifold.
By considering the long exact sequence of the pair $(W, \partial W)$, we obtain that
\[H^3(\wt{W}; \mathbb{F}_p) \cong H_1(\wt{W}, \partial \wt{W}; \mathbb{F}_p) \cong \mathbb{F}_p,\]
and hence that $b_2^{p}(\wt{W})=0$. It then follows from the same long exact sequence that $b_1^p(\partial \wt{W})= b_1^p( Y_n \sqcup S^3)=0$, and so we have established the first claim.

Let $E(K)$ denote the exterior of the knot $K$ in $S^3$ and $\mu_K$ and $\lambda_K$ denote the meridian and longitude respectively. By definition the manifold $S^3_n(K)= E(K) \cup (S^1 \times D^2)$, where $\{\pt\} \times \partial D^2$ is identified with $n \mu_K + \lambda_K$ and $S^1 \times \{\pt\}$ with $\lambda_K$. Therefore, we have that
\[Y_n= ( E(K) \cup (S^1 \times D^2))_n= E_n(K) \cup (S^1 \times D^2),\]
where $\{\pt\} \times \partial D^2$ is identified with $\wt{\mu_K^n}+ \wt{\lambda}_K$  and $S^1 \times \{\pt\}$ with $\wt{\lambda}_K$, where $\widetilde{\cdot}$ denotes a lift to the $n$-fold cover $E_n(K)\to E(K)$.

We also have that
\[\Sigma_n(K)=  E_n(K) \cup (S^1 \times D^2)\]
where $\{\pt\} \times \partial D^2$ is identified with $\wt{\mu_K^n}$  and $S^1 \times \{\pt\}$ with $\wt{\lambda}_K$.
Since $\wt{\lambda}_K$ is null-homologous in $E_n(K)$, we see that $H_1(\Sigma_n(K)) $ and $H_1(Y_n)$ are isomorphic quotients of $H_1(E_n(K))$.
\end{proof}

We extract the next statement from the proof of Proposition~\ref{prop:trivialh1} for later use.

\begin{corollary}\label{Scholium:matthias-fault}
%Suppose that $K$ is $\Z/n$-shake slice via an embedded sphere $S$ in $X_n(K)$
Let $K$ be a knot such that $S^3_n(K)$ and $L(n,1)$ are homology cobordant via a homology cobordism $W$ with $\pi_1(W) \cong \Z/n$ and let $\wt{W}$ denote the $\Z/n$-cover of $W$.
%= X_n(K) \smallsetminus \nu(S)$.
Then  $H_2(\wt{W};\mathbb{F}_p)=0$ for every prime $p$.
\end{corollary}

We would now like to prove the following. Note that the first statement implies that if a knot $K$ is $n$-shake slice then $\sigma_{\omega}(K)=0$ for every $q^{th}$ root of unity $\omega$, where $q$ is a prime power dividing $n$. This was originally proved using different methods by Tristram~\cite{tristram}.

\begin{prop}\label{prop:signaturesobstruct}
%Suppose that $K$ is $n$-shake slice.
Let $K$ be a knot such that $S^3_n(K)$ and $L(n,1)$ are homology cobordant via a homology cobordism $W$.
Then $\sigma_{\omega}(K)=0$ for every $q^{th}$ root of unity $\omega$, where $q$ is a prime power dividing $n$. If
%$K$ is moreover $\Z/n$-shake slice,
$\pi_1(W) \cong \Z/n$, then $\sigma_{\omega}(K)=0$ for every $n^{th}$ root of unity $\omega$.
\end{prop}

Our strategy in proving Proposition~\ref{prop:signaturesobstruct} will be to relate the Tristram-Levine signature of $K$ at the $n^{th}$ roots of unity to the Atiyah-Singer/Casson-Gordon signatures~\cite{Atiyah-Singer:1968-3, CassonGordon78} of the 3-manifold $S^3_n(K)$.  We will then use the $n$-fold cyclic cover of the hypothesised  homology cobordism $W$ between $S^3_n(K)$ and $L(n,1)$, capped off in a certain nice way, to compute these Casson-Gordon signatures. We therefore recall the definition of the Casson-Gordon signatures of a 3-manifold, as given in \cite{CassonGordon78}.
For every $n \in \mathbb{N}$ we think of the cyclic group~$\Z/n$ as coming with a canonical multiplicative generator $t$.

To a closed oriented 3-manifold $Y$ and a map $\phi \colon H_1(Y) \to \Z/n$, we wish to associate $\sigma_k(Y, \phi) \in \Q$ for $k=1, \dots, n-1$.
Let $\wt{Y} \to Y$ be the covering induced by $\phi$, and note that there is a canonical covering transformation $\tau$ of $\wt{Y}$ corresponding to $t \in \Z/n$.

Now suppose there exists a $\Z/n$ branched covering of 4-manifolds $\wt{Z} \to Z$, branched over a surface $F$ contained in the interior of $Z$, and such that $\partial(\wt{Z} \to Z)=(\wt{Y} \to Y)$. Suppose in addition that the covering transformation $\wt{\tau}\colon\wt{Z}\to\wt{Z}$ that induces rotation through $2 \pi/n$ on the fibres of the normal bundle of $\wt{F}$ is such that it restricts on $\wt{Y}$ to $\tau$. An explicit construction as given, for example, in the proof of \cite[Lemma~3.1]{CassonGordon78}, shows that such a branched cover does always exist.
Then $H_2(\wt{Z};\mathbb{C})$ decomposes as $\bigoplus_{k=0}^{n-1} V_k$, where $V_k$ is the $\xi_n^k$-eigenspace of the $\wt{\tau}$-induced action on second homology, and as before $\xi_n := e^{2 \pi i/n}$.
Let $\sigma_k(\wt{Z})$ denote the signature of the intersection form of $\wt{Z}$ restricted to $V_k$. Define, for $k=1, \dots, n-1$, the signature defect
\begin{equation}\label{CG-formula}
\sigma_k(Y, \phi)  := \sigma(Z)- \sigma_k(\wt{Z}) - \frac{2 ([F]\cdot [F]) k (n-k)}{n^2}.
  \end{equation}

Casson-Gordon~\cite{CassonGordon78} used the Atiyah-Singer $G$-signature theorem~\cite{Atiyah-Singer:1968-3} to show that $\sigma_k(Y, \phi)$ is an invariant of the pair $(Y,\phi)$ for each $0<k<n$.

\begin{prop}\label{prop:TlsigsandCGsigs}
Let $K$ be a knot and $W$ be a cobordism between $S^3_n(K)$ and $L(n,1)$ over $\Z/n$.
Let $\xi_n=e^{2 \pi i/n}$ and $1 \leq k \leq n-1$.
Then
\[\sigma_{\xi_n^k}(K)= \sigma_k(\widetilde{W})- \sigma(W),\]
where $\sigma_k(\widetilde{W})$ is the signature of the intersection form of the $\Z/n$-cover of $W$ induced by the map $H_1(W) \to \Z/n$ when restricted to the $\xi_n^k$-eigenspace of the action of the generator of the group of deck transformations on $H_2(\widetilde{W}, \mathbb{C})$.
\end{prop}
\begin{proof}
Our proof follows from computing $\sigma_k(S^3_n(K), \phi)$, where $\phi \colon H_1(S^3_n(K)) \to \Z/n$ is the canonical map sending the meridian $\mu_K$ to $1 \in \Z/n$, in two different ways.

First, observe that in this setting the surgery formula of \cite[Lemma~3.1]{CassonGordon78} is particularly simple and reduces to
\begin{align}\label{formula:cg1}
\sigma_k(S^3_n(K), \phi)= 1- \sigma_{\xi_n^k}(K)-\frac{2k(n-k)}{n}.
\end{align}
Secondly, let $\wt{W} \to W$ be the $n$-fold cyclic cover. We have that
\[\partial(\wt{W} \to W)= (\widetilde{S^3_n(K)} \to S^3_n(K)) \sqcup (S^3 \to L(n,1)).\]
Now let $X_n(U)$ denote the $n$-trace of the unknot, i.e.\ the disc bundle $D_n$ over $S^2$ with euler number $n$. Let $S$ be the $n$-framed embedded 2-sphere in $X_n(U)$.
There is an $n$-fold cyclic branched cover $\wt{X}_S$ of $X_n(U)$ along $S$, which restricts on the boundary to the same (unbranched) cover $S^3 \to L(n,1)$ we saw above. Note that $\wt{X}_S$ is a punctured $\mathbb{CP}^2$.

We can therefore use $Z:= W \cup_{L(n,1)} X_n(U)$ and $\widetilde{Z}= \wt{W} \cup_{S^3} \wt{X}_S$ to compute $\sigma_k(S^3_n(K), \phi)$ using (\ref{CG-formula}).
Note that $H_2(\widetilde{Z}) \cong H_2(\widetilde{W}) \oplus \Z$, where the generator of the $\Z$ summand is represented by the lift of $S$ and hence has self-intersection $+1$, intersects trivially with all elements of $H_2(\widetilde{W})$, and  is preserved under the action of the covering transformation so lies in the $1$-eigenspace. Consequently,   for $0<k<n$ we have $\sigma_k(\wt{Z})=\sigma_k(\wt{W})$
and hence
\begin{align}
\sigma_k(S^3_n(K), \phi)&= \sigma(Z)- \sigma_k(\wt{Z}) - \frac{2 ([S] \cdot [S]) k(n-k)}{n^2} \nonumber \\
&= (\sigma(W)+1)- \sigma_k(\wt{W}) - \frac{2nk(n-k)}{n^2}. \label{formula:cg2}
\end{align}

Since $\sigma_k(S^3_n(K), \phi)$ is well-defined, by comparing the formulae of Equations~\ref{formula:cg1} and~\ref{formula:cg2} we  obtain as desired that
\[\sigma_{\xi_n^k}(K)= \sigma_k(\widetilde{W})- \sigma(W). \qedhere\]
\end{proof}

\begin{lem}\label{lem:lifting-homology-equiv}
Let $i \colon X \to Y$ be a map of $($spaces homotopy equivalent to$)$ finite CW complexes
that induces isomorphisms $i_* \colon H_k(X; \Z) \to H_k(Y;\Z)$ for all~$k$.
Suppose $\varepsilon \colon H_1(Y;\Z) \to \Z/{q}$ is a surjective map inducing $\Z/q$-covers $\wt{Y} \to Y$ and $\wt{X} \to X$.

If $q$ is a prime power, then the induced map
\[\wt{i}_* \colon H_k(\wt{X}; \Q) \to H_k(\wt{Y};\Q)\]
is an isomorphism for all~$k$.
\end{lem}
\begin{proof}
Let $\alpha \colon H_1(Y) \to \GL_q(\Q)$ be the map obtained by composing $\varepsilon$ with the regular representation
\begin{align*}
        \Z/q &\to \GL_q(\Q)   \\
           k &\mapsto
\left[ \begin{array}{ccccc}
0&1&0&\dots &0 \\
0 & 0 & 1 & \dots & 0 \\
\vdots & \vdots & \ddots & \ddots & \vdots \\
0& 0 & 0 & \dots & 1 \\
1 & 0 & 0 & \dots & 0
\end{array} \right]^k.
 \end{align*}
As usual, this regular representation endows $\Q^q$ with the structure of a free $\Q[\Z/q]$-module of rank one.

By Friedl-Powell~\cite[Proposition~4.1]{Friedl-Powell:2010-1}, (applied with, in their notation, $H=\{1\}$), we have that
\[i_* \colon H_*(X; \Q^q) \to H_*(Y; \Q^q)\]
is an isomorphism.  Then for $Z\in \{X,Y\}$, we have natural identifications
\begin{align*}
H_*(Z; \Q^q)= H_*(C_*(\wt{Z}, \Q) \otimes_{\Q[\Z/q]} \Q^q)=H_*(C_*(\wt{Z}, \Q) \otimes_{\Q[\Z/q]} \Q[\Z/q])=H_*(\wt{Z}, \Q),
\end{align*}
and the desired result follows.
\end{proof}

%Now we  prove Proposition~\ref{prop:signaturesobstruct}.

\begin{proof}[Proof of Proposition~\ref{prop:signaturesobstruct}.]
%As usual, let $W= X_n(K) \smallsetminus \nu(S)$, where $S$ is the hypothesised embedded sphere generating $\pi_2(X_n(K))$, and
Let $\wt{W}$ denote the $\Z/n$ cover of $W$. Note that since $W$ is a homology cobordism between $S^3_n(K)$ and $L(n,1)$ we have that $H_2(W)=0$ and so certainly $\sigma(W)=0$.

The case when
%$K$ is $\Z/n$-shake slice
$\pi_1(W)\cong \Z/n$
follows quickly: Corollary~\ref{Scholium:matthias-fault} tells us that $H_2(\wt{W}; \mathbb{F}_p)=0$ for all primes $p$, and hence that $H_2(\wt{W}; \Z)=0$ and so $H_2(\wt{W}; \mathbb{C})=0$. Therefore, by Proposition~\ref{prop:TlsigsandCGsigs} we have for $k=1, \dots, n-1$ that
\[\sigma_{\xi_n^k}(K)= \sigma_k(\widetilde{W})- \sigma(W)=0-0=0.\]

So we now assume only that
$W$ is a homology cobordism, with no condition on the fundamental group.
%$K$ is $n$-shake slice,
Let $q$ be a prime power dividing $n$, and let $1 \leq k \leq q-1$ be relatively prime to $q$.
Let $\phi \colon H_1(S^3_n(K)) \to \Z/q$ be the map sending the class $[\mu_K]$ of the meridian of $K$ to $+1 \in \Z/q$.
We now argue exactly as in the proof of Proposition~\ref{prop:TlsigsandCGsigs} to show that
\[(\sigma(W)+1)- \sigma_k(\wt{W}) - \frac{2k(q-k)}{q}= \sigma_k(S^3_n(K), \phi)= 1- \sigma_{\xi_q^k}(K)-\frac{2k(q-k)}{q}\]
and hence, since $H_2(W; \mathbb{Q})=0$, that
\[ \sigma_{\xi_q^k}(K)= \sigma_k(\wt{W})- \sigma(W)=  \sigma_k(\wt{W}).\]

%\fotnote{Extended version, to replace the previous sentence if we like:\\
% As in the proof of Proposition~\ref{prop:TlsigsandCGsigs}, we apply Lemma 3.1 of \cite{CassonGordon78} to obtain that
%\[\sigma_k(S^3_n(K), \phi)= 1- \sigma_{\xi_q^k}(K)-\frac{2k(q-k)}{q}.\]
%Now let $W$ be the exterior of the $n$-shake sphere.
%As in the proof of Proposition~\ref{prop:TlsigsandCGsigs}, we apply Equation~\ref{CG-formula} to the union of $W$ and $X_n(U)$, and compute
%\[ \sigma_K(S^3_n(K), \phi)= (\sigma(W)+1)- \sigma_k(\wt{W}) - \frac{2k(q-k)}{q}.\]
%We thereby obtain
%$\sigma_{\xi_q^k}(K)= \sigma_k(\wt{W})- \sigma(W)=  \sigma_k(\wt{W})$.}

But by Lemma~\ref{lem:lifting-homology-equiv}, since the inclusion induced map
\[i_* \colon H_*(L(n,1); \Z) \to H_*(W; \Z)\]
is an isomorphism and $q$ is a prime power, we have that
\[\wt{i}_* \colon H_*(\wt{L(n,1)}; \Q) \to H_*(\wt{W}; \Q)\]
is also an isomorphism. But since $\wt{L(n,1)}$ is itself a lens space (or $S^3$ if $n=q$), we have that $H_2(\wt{L(n,1)}; \Q)=0$ and so  $H_2(\wt{W}; \Q)=0$ as well. Thus the $\xi_n^k$-eigenspace $V_k=0$, and so as desired
\[ \sigma_{\xi_q^k}(K)= \sigma_k(\wt{W})=0.\qedhere \]
\end{proof}

\section{Setting up the surgery problem}\label{sec:surgeryproblem}

We will use surgery theory to construct the exterior of the desired embedded sphere in an $n$-trace. We will eventually apply surgery in the topological category, but our initial input manifolds will be smooth. We thus now recall the input to a \emph{surgery problem} in the smooth category. There is an analogous theory in the topological category, and we will discuss this below at the point when it becomes necessary.

\begin{definition}\label{def:smoothnormal} Given a smooth $m$-manifold $X$, the tangent bundle is classified up to isomorphism by a homotopy class of maps $\tau_X\colon X\to \BO(m)\subset \BO$. The unique stable bundle $\nu_X\colon X\to \BO$ such that $\tau_X\oplus\nu_X\colon X\to \BO$ is null-homotopic is called the \emph{stable normal bundle} of $X$. The manifold $X$ can be \emph{stably framed} if $\nu_X$ is null-homotopic, and a choice of null homotopy is called a \emph{stable framing}. A choice of stable framing for $X$ is determined by a choice of stable trivialisation of the tangent bundle; that is, a choice of $k$ and vector bundle isomorphism $TX\oplus \underline{\R}^k\cong\underline{\R}^{m+k}$.
\end{definition}

Recall that an oriented $m$-manifold $X$ with (possibly empty) boundary has a \emph{fundamental class}, denoted $[X,\partial X]\in H_m(X,\partial X;\Z)$,  capping with which induces (twisted) Poincar\'{e}-Lefschetz duality isomorphisms $H^{m-k}(X,\partial X;\Z[\pi_1(X)]) \xrightarrow{\cong} H_{k}(X;\Z[\pi_1(X)])$  and  $H^{m-k}(X;\Z[\pi_1(X)]) \xrightarrow{\cong} H_{k}(X,\partial X;\Z[\pi_1(X)])$ for every $k$.

\begin{definition}\label{defn:DONM}
A map $(f,\partial f)\colon (X,\partial X)\to (Y,\partial Y)$ of smooth oriented $m$-manifolds with (possibly empty) boundary is called \emph{degree one} if $f_*([X, \partial X])=[Y,\partial Y]$. Given a degree one map $f$, a \emph{normal structure} is an isomorphism of stable bundles $\nu_X\simeq\nu_Y\circ f$.
A degree one map with choice of normal structure is called a \emph{degree one normal map}. We will often write $(X,f)$ for the data of a degree one normal map (suppressing the choice of stable bundle isomorphism).

For a topological space $Y$, a bordism $Z$ between closed $m$-manifolds $X$ and $X'$ is \emph{over Y} if there is a proper map $F\colon Z\to Y\times I$ such that $F(X)\subset Y\times\{0\}$ and $F(X')\subset Y\times\{1\}$. If $Y$ and $Z$ are smooth oriented $m$-manifolds and the map $F$ is a degree one normal map, then we call $(Z,F)$ a \emph{degree one normal bordism} from $(X,F|_X)$ to $(X',F|_{X'})$.
\end{definition}

\begin{remark}\label{rem:stabframe}
Given a degree one map $(f,\partial f)\colon (X,\partial X)\to (Y,\partial Y)$, if $\nu_Y$ is null-homotopic, then so is $\nu_Y\circ f$. So $f$ admits a normal structure if and only if $X$ can be stably framed.

Furthermore, we will sometimes be interested in picking a normal structure on $f$ that is compatible with a given one on $\partial f$. To understand this, suppose we are given a choice of stable framing on $Y$. This induces a choice of stable framing on $\partial Y$. Suppose we have a degree one map $\partial f\colon \partial X\to\partial Y$. A choice of normal structure on $\partial f$ is equivalent to a choice of stable framing on $\partial X$. Suppose such a choice has been made. Then a degree one map $(f,\partial f)\colon (X,\partial X)\to (Y,\partial Y)$ admits a normal structure inducing the given one on the boundary if and only if $X$ can be stably framed compatibly with $\partial X$.
\end{remark}

Note that the lens space $L(n,1)$ is diffeomorphic to the result of $n$-surgery along the unknot. In our applications, the target manifold for the map in the surgery problem will be either $(L(n,1),\varnothing)$ or $(L(n,1) \times I,L(n,1) \times \{0,1\})$. The tangent bundles of $L(n,1)$ and of $L(n,1) \times I$ are trivial, so in particular these manifolds can be stably framed. Choose once and for all a stable framing for $L(n,1)$, and hence one for $L(n,1) \times I$.

\begin{lemma}\label{lem:deg-one-map}
For any knot $K$, there exists a degree one normal map $f \colon S_n^3(K) \to L(n,1)$ that is a $\Z$-homology equivalence and that extends to a homotopy equivalence $\overline{f}\colon X_n(K) \to D_n$, where $D_n$ is the $D^2$-bundle over $S^2$ with euler number $n$. Additionally, one can arrange that the cocore of the $2$-handle of $X_n(K)$ maps to the cocore of the $2$-handle in the standard handle decomposition for $D_n$.
\end{lemma}

\begin{proof}
There is a standard degree one map from $E(K) \to E(U)$ that realises the homology equivalence $E(K)\to S^1$ and is the identity map on the boundary; see e.g.~\cite[Construction 7.1]{Miller-Powell} for the details. As submanifolds of $S^3$, both $E(K)$ and $E(U)$ have stably trivial tangent bundles and can thus be stably framed. So this degree one map can be given a choice of normal structure by Remark \ref{rem:stabframe}. Now extend this degree one normal map to the Dehn filling.

We construct the homotopy equivalence $\overline{f}$. By construction, $f\colon S^3_n(K)\rightarrow S^3_n(U)=L(n,1)$ sends a meridian of $K$ to a meridian of the unknot $U$. We also have a handle decomposition of $X_n(K)$ relative to its boundary consisting of a $2$-handle, attached along a meridian of $K$, and a $4$-handle. Define $\overline{f}\colon X_n(K)\to X_n(U)=D_n$ by mapping the $2$-handles and the $4$-handles homeomorphically to each other. To do this, first note that the attaching circle of the 2-handle of $X_n(U)$ is a meridian of $U$, so the attaching circle of the 2-handle of $X_n(K)$ is sent to the attaching circle of the 2-handle of $X_n(U)$. Also the framings agree, so we can extend over the 2-handle.  Then note that attaching the 2-handles undoes the Dehn surgeries, converting both $S^3_n(K)$ and $S^3_n(U)$ to $S^3$. The 4-handles are attached to these copies of $S^3$. Since every orientation-preserving homeomorphism of $S^3$ is isotopic to the identity, we may extend the map over the 4-handles.  Observe that $\overline{f}$ is a homotopy equivalence by Whitehead's theorem.
\end{proof}

Recall the Whitehead group $\Wh(\Z/n)$ is a certain quotient of the algebraic $K$-group $K_1(\Z[\Z/n])$. When $H_1(S^3_n(K);\Z[\Z/n])=0$, a map $f$ as in Lemma~\ref{lem:deg-one-map} induces a chain homotopy equivalence $\widetilde{f}_*\colon C_*(S^3_n(K);\Z[\Z/n])\to C_*(L(n,1);\Z[\Z/n])$ and thus determines a $\Z[\Z/n]$-coefficient Whitehead torsion $\tau(f):=\tau(\widetilde{f}_*) \in\Wh(\Z/n)$. We will need the following technical lemma later.

\begin{lemma}\label{lem:whitehead}If the knot $K$ satisfies $H_1(S^3_n(K);\Z[\Z/n])=0$, then a map $f$ as in Lemma~\ref{lem:deg-one-map} has trivial $\Z[\Z/n]$-coefficient Whitehead torsion $\tau(f)=1\in\Wh(\Z/n)$.
\end{lemma}

\begin{proof} Fix a cell complex $A\cong S^1\times S^1$ and extend it to a cell complex $B\cong S^1\times D^2$. Denote by $Y(K)$ and $Y(U)$ any fixed choice of cell structures for $E(K)$ and $E(U)$ (we will make quite specific choices later). For each of $J=U, K$, let $M(J)$ denote the cell complex obtained as the mapping cylinder of a map $A\to Y(J)$ which is a cellular approximation to the inclusion $S^1\times S^1\to E(J)$.

Each of these spaces has a $\Z/n$ cover, determined by $\pi_1(L(n,1))\cong\Z/n$ and composing with, where appropriate, the map $f$. Choose a lift of the cell structure on $A$ to the $\Z/n$ cover, and denote this by $\widetilde{A}$. Extend this lift to $B$, $M(J)$ and $M(U)$ and write $\wt{B}$, $\wt{M(J)}$ and $\wt{M(U)}$ for the corresponding covering spaces. Write $\wt{f}$ for a lift of $f$ to the $\Z/n$ covers.
\[
\begin{tikzcd}
0 \arrow{r} &C_*(\wt{A}) \arrow{r} \arrow{d}{\id} & C_*(\wt{B})\oplus C_*(\wt{M(K))} \arrow{r} \arrow{d}{\id\oplus (\wt{f}\vert_{\wt{M(K)}})_*} & C_*(\wt{S^3_n(K)}) \arrow{r} \arrow{d}{\wt{f}_*} \arrow{r} & 0\\
0 \arrow{r} &C_*(\wt{A}) \arrow{r} \arrow{d} & C_*(\wt{B})\oplus C_*(\wt{M(U)}) \arrow{r} \arrow{d} & C_*(\wt{L(n,1)}) \arrow{r} \arrow{d} & 0\\
0 \arrow{r} & \mathscr{C}(\Id) \arrow{r} & \mathscr{C}(\Id) \oplus \mathscr{C}((\wt{f}\vert_{\wt{M(K)}})_*) \arrow{r} & \mathscr{C}(\wt{f}_*) \arrow{r} & 0.
\end{tikzcd}
\]
In this diagram we use cellular chain complexes and $\mathscr{C}$ denotes taking an algebraic mapping cone. Strictly, we have replaced the maps $\wt{f}$ and $\wt{f}|_{\wt{M(K)}}$ by cellular approximations. This is thus a diagram in the category of finite, finitely generated, based $\Z[\Z/n]$-coefficient chain complexes. All complexes in the lower sequence are acyclic so by the multiplicativity of torsion under such exact sequences, and the fact that the identity map has vanishing torsion, we obtain
$
\tau(f)=\tau(\wt{f}_*)=\tau((\wt{f}|_{\wt{M(K)}})_*)\in \Wh(\Z/n).
$

We now choose convenient cell structures for $E(K)$ and $E(U)$. Let $Y(K)$ and $Y(U)$ be the cell structures associated to Wirtinger presentations of the respective knot groups (see e.g.~proof of~\cite[Theorem 16.5]{Turaev:2001-1}), where for $K$ we choose an arbitrary such presentation and for $U$ we choose the presentation of~$\Z$ with no relations. By~\cite[Lemma 8.4]{Turaev:2001-1}, the Whitehead torsions $\tau(M(J),Y(J))=1$ for each of $J=K,U$, so we have that $\tau((\wt{f}|_{\wt{M(K)}})_*)$ may be computed by a cellular map $F\colon \wt{Y(K)}\to \wt{Y(U)}$ representing $\wt{f}|_{\wt{E(K)}}$. This map is given by
\[
\begin{tikzcd}
C_*(Y(K))\ar[d,"F"]&0\ar[r] &\bigoplus_{j=1}^m\Z[\Z/n] \ar[r, "\partial_2(t)"]\ar[d] &\bigoplus_{i=1}^{m+1}\Z[\Z/n] \ar[r, "t-1"]\ar[d,"(1\,\ldots \,1)"] &\Z[\Z/n] \ar[r]\ar[d, "1"] &0\\
C_*(Y(U))&0\ar[r] &0\ar[r] &\Z[\Z/n] \ar[r,"t-1"] &\Z[\Z/n] \ar[r] & 0.
\end{tikzcd}
\]
where we are writing $\Z[\Z/n]=\Z[t, t^{-1}]/(t^n-1)$, $m+1$ is the number of generators in the Wirtinger presentation for $\pi_1(K)$, and $\partial_2(t)$ is determined by the relations in that presentation. Up to basis change, and a simple homotopy equivalence, the algebraic mapping cone  $\mathscr{C}(F)$ is
\[
\begin{tikzcd}
0\ar[r] &\bigoplus_{j=1}^m\Z[\Z/n]\ar[r, "\partial'_2(t)"] &\bigoplus_{i=1}^{m}\Z[\Z/n] \ar[r]&0\ar[r] &0\ar[r]&0
\end{tikzcd}
\]
where $\partial'_2(t)$ is the effect of deleting a row from the matrix $\partial_2(t)$, after the basis change. We then compute that $\tau((\wt{f}|_{\wt{M(K)}})_*)=\tau(F)=\left[ \partial'_2(t)\right]\in \Wh(\Z/n)$.

When $G$ is an abelian group, the determinant map $\det\colon K_1(\Z[G])\to (\Z[G])^\times$ is a split surjection (see e.g.~\cite[p.~359]{MR196736}), and for any finite cyclic $G$ the kernel of this map vanishes~\cite[Theorem 5.6]{MR933091}. So the determinant gives an isomorphism $K_1(\Z[\Z/n])\cong (\Z[\Z/n])^\times$. The matrix $\partial_2'(t)$ is a presentation matrix for the Alexander module of $K$ when considered over the ring $\Z[t,t^{-1}]$. From this, and using the left regular representation of $\Z[\Z/n]$, we may consider $\partial'_2$ as an $mn\times mn$ matrix over~$\Z$, such that the absolute value of the determinant $|\det(\partial'_2)|$ is the order of the first homology of the $n$-fold branched cover $H_1(\Sigma_n(K);\Z)$.
 This, in turn, is the order of $H_1(S^3_n(K);\Z[\Z/n])$, which we have assumed to be $1$. Thus, as an element of $K_1(\Z[\Z/n])\cong (\Z[\Z/n])^\times$, we have that $\tau(F)=\det(\partial'_2(t))=\pm 1$, under this isomorphism. Finally, both $+1$ and $-1$ become the trivial element on passage to the Whitehead group $\Wh(\Z/n)$, so we obtain the desired result.
\end{proof}

\begin{lemma}\label{lem:normal-bordism-even-n}
When $\Arf(K)=0$, there exists a degree one normal  map $(S_n^3(K),f)$ satisfying the conditions of Lemma \ref{lem:deg-one-map}, and that is degree one normal bordant over $L(n,1)$ to the identity map $(L(n,1),\Id)$.
\end{lemma}

\begin{proof} Write $f \colon S_n^3(K) \to L(n,1)$ for the degree one normal map obtained in Lemma  \ref{lem:deg-one-map}. By Remark \ref{rem:stabframe}, if we can show there is a degree one map $F\colon W\to L(n,1)\times I$ describing a stably framed cobordism over $L(n,1)$ from $(S_n^3(K),f)$ to $(L(n,1),\Id)$, we will be done. (Note, we may change the choice of stable framing on $S_n^3(K)$ during the course of the proof.)
%\fotnote{AM: And at this point in the proof we have not yet chosen any stable framing on $S_n^3(K)$? PO: we have chosen one - a choice of normal data is equivalent to a choice of stable framing}

Consider the closed 3-manifold \[N_K:= E(K) \cup_{S^1 \times S^1} S^1 \times S^1 \times I   \cup_{S^1 \times S^1} -E(U),\] glueing so that the result is homeomorphic to the 0-framed surgery $S^3_0(K)$, but we decompose it in this way for later. Choose a framing of the tangent bundle of $N_K$ that is a product framing on $S^1 \times S^1 \times I$.  This determines an element of $[N_K,g]\in\Omega_3^{\fr}(S^1 \times D^2\times I)$, where the map $g\colon N_K\to S^1 \times D^2 \times I = E(U) \times I$ has image
\[\partial (S^1 \times D^2 \times I) = S^1 \times D^2 \cup S^1 \times S^1 \times I \cup - S^1 \times D^2,\]
and is obtained by glueing together the standard degree one map $E(K) \to E(U) = S^1 \times D^2$, the identity on $S^1 \times S^1 \times I$, and the canonical  identification $E(U) \to S^1 \times D^2$.

We have $\Omega_3^{\fr}(S^1 \times D^2\times I) \cong \Omega^{\fr}_3(S^1)$. The map from $S^1$ to a point induces a split short exact sequence $0\to \widetilde{\Omega}_3^{\fr}(S^1)\to \Omega^{\fr}_3(S^1)\to \Omega^{\fr}_3\to 0$. Thus $\Omega^{\fr}_3(S^1)\cong \Omega^{\fr}_3\oplus \Omega^{\fr}_2$, where we have used the isomorphism $\widetilde{\Omega}^{\fr}_3(S^1)\cong \widetilde{\Omega}^{\fr}_2(S^0)\cong \Omega^{\fr}_2$, coming from the fact that $\Omega_*^{\fr}(-)$ is a generalised homology theory and thus satisfies the suspension axiom.  Under this isomorphism, the projection of $[N_K,g]$ to $\Omega_3^{\fr}\cong \Z/24$ is given by the framed bordism class of $N_K$.
 Since any element in $\Omega_3^{\fr}\cong \Z/24$ can be realised by a suitable framing on~$S^3$, the framing of $N_K$ can be modified in a small neighbourhood until the element in $\Omega_3^{\fr}\cong \Z/24$ vanishes; see e.g.~\cite[p.~13]{Cha-Powell-2014-1} for details.

The map from $\Omega_2^{\fr}\cong\widetilde{\Omega}_2^{\fr}(S^0)\cong \widetilde{\Omega}_3^{\fr}(S^1)\to\Omega_3(S^1)$ is given by applying the reduced suspension, sending the class of $f\colon F^2\to \pt$ to $ f\times \Id\colon F\times S^1\to S^1$ in $\Omega_3(S^1)$, and framing $F\times S^1$ via the product framing with the trivial framing on $S^1$. Under the projection $\Omega_3(S^1)\to \Omega_2^{\fr}$, the class $[N_K,g]$ is thus mapped to a framed surface $F$ given by taking the transverse preimage of a regular point in $S^1$ under $g$. Recall that $g$ was obtained using the standard degree one normal maps $E(K)\to E(U)$ and $E(U)\to E(U)$. Under these, generic transverse point preimages are Seifert surfaces, $F_K$ and $F_U$, respectively for the knots $K$ and $U$. Thus the projection to $\Omega_2^{\fr}$ is the Arf invariant of the framed surface $F=F_K \cup (\{\pt\} \times S^1 )\cup - F_U$. Since we assumed that $\Arf(K)=0$ this component already vanishes in $\Omega^{\fr}_2$. So overall we obtain $[N_K,g]=0\in\Omega_3^{\fr}(S^1 \times D^2\times I)$

Write $W'$ for a framed null bordism of $(N_K,g)$ over $S^1 \times D^2\times I$. Now attach  $D^2\times S^1 \times I$ to $S^1 \times S^1 \times I$ such that $D^2\times S^1 \times \{t\}$ attaches to $S^1 \times S^1 \times \{t\}$ with the $n$-framing, and similarly attach the $n$-framed  $D^2\times S^1 \times I$ to the codomain $S^1\times D^2 \times I$.
This yields a map $F\colon W\to L(n,1)\times I$ that describes a framed cobordism from $(S_n^3(K),f)$ to $(L(n,1),\Id)$, over $L(n,1)$.  To see that the cobordism is framed, note that we can glue the two framings together along the product framing on both pieces $W'$ and $D^2 \times S^1 \times I$.

We finally also note that the map $F\colon W\to L(n,1)$ is degree one. This can be computed by considering the naturality of the long exact sequences in homology of the pairs $(W,\partial W)$ and $(L(n,1) \times I,L(n,1) \times \{0,1\})$, and the fact that both components of $\partial F$ are already known to be degree one.
\end{proof}

\begin{remark}\label{prop:Brown-Kervaire}
For even $n$, there is a sequence of group homomorphisms
\[
\begin{tikzcd}
\Omega_3^{\fr}(L(n,1))\arrow[r] & \Omega_3^{\fr}(B(\Z/n)) \arrow[r]&\Omega_3^{\fr}(B(\Z/2))\arrow[r] &\Omega_3^{\operatorname{Spin}}(B(\Z/2))
\end{tikzcd}
\]
given from left to right by: the inclusion of $L(n,1)$ as the $3$-skeleton, the surjective group homomorphism $\Z/n\to \Z/2$, and the forgetful map.

An argument similar to that of \cite[Lemma 4.2]{MR3032093} can be made to determine that $\Omega_3^{\operatorname{Spin}}(B(\Z/2))\cong \Omega_2^{\operatorname{Pin}^-}\cong\Z/8$. This argument would take us too far afield here, but we allow ourselves to consider the consequences. Under this isomorphism, the element of $\Z/8$ is detected by the Brown invariant of a $\operatorname{Pin}^{-}$ structure on the surface in $M$ that is the transverse preimage of $\mathbb{RP}^2 \subseteq \mathbb{RP}^3$, the $3$-skeleton of $B(\Z/2)$. Let $(M,f)\in \Omega_3^{\operatorname{Spin}}(B(\Z/2))$ with $f^{-1}(\RP^2)=N\subset M$. The $\operatorname{Pin}^{-}$ structure on $N$ gives a $\Z/4$ enhancement of the $\Z/2$ intersection form on $N$, which counts the number of half twists modulo 4 of the bands. If the surface $N$ is orientable, then the enhancement lies in $2\Z/4$ and the  Brown invariant lies in $4\Z/8\cong \Z/2$, computing the Arf invariant of the surface $N$ with respect to the $\operatorname{Spin}$ structure on $N$ pulled back from $M$ \cite[\textsection 3]{MR1171915}.

In our case, the map $f$ is the one from Lemma \ref{lem:deg-one-map}, and is constructed as a Pontryagin-Thom collapse on a normal neighbourhood of a Seifert surface, followed by the extension to the Dehn filling. Thus in this case, the surface $N=f^{-1}(\RP^2)\subset S^3_n(K)$ is represented by a capped off Seifert surface for $K$. This shows that the obstruction to vanishing of $(S^3_n(K),f)\in \Omega_3^{\operatorname{Spin}}(B(\Z/2))$ is given by the Arf invariant of the knot.
%\fotnote{AM: It seems like both of the previous two paragraphs end by concluding that $(S^3_n(K),f)=0$ if and only if $\Arf(K)=0$. Is that right? PO: Not quite. The first paragraph concludes that Arf invariant of the surface vanishes. The second paragraph argues that this Arf invariant just computed is the Arf invariant of the knot.}
\end{remark}

For even $n$, Remark \ref{prop:Brown-Kervaire} shows that the bordism of Lemma \ref{lem:normal-bordism-even-n} cannot exist unless $\Arf(K)=0$. This is not so for $n$ odd, and to demonstrate this we now include an alternative existence proof for degree one normal bordism, with no Arf invariant assumption, when $n$ is odd.

\begin{lemma}\label{lem:normal-bordism-odd-n}
For $n$ odd, there exists a degree one normal  map $(S_n^3(K),f)$ that satisfies the conditions of Lemma \ref{lem:deg-one-map} and is degree one normal bordant over $L(n,1)$ to the identity map $(L(n,1),\Id)$.
\end{lemma}

\begin{proof} Write $f \colon S_n^3(K) \to L(n,1)$ for the degree one normal map obtained in Lemma  \ref{lem:deg-one-map}. Again, by Remark \ref{rem:stabframe}, the objective is to show there is a degree one map $F\colon W\to L(n,1)\times I$ describing a stably framed cobordism over $L(n,1)$ from $(S_n^3(K),f)$ to $(L(n,1),\Id)$.

We show that there is a choice of stable framing for $S_n^3(K)$ such that $[S_n^3(K),f] - [L(n,1),\Id] =0 \in \Omega^{\fr}_3(L(n,1))$.

%First take the CW complex for $L(n,1)$, with one cell $e^p$ in each of the dimensions $p=0, 1, 2, 3$, and consider the obstructions to modifying $(Y_K,g)$ within $\Omega^{\fr}_3(L(n,1))$, so that the image of the map to $L(n,1)$ lies in the $p$-skeleton $X^{(p)}\subset L(n,1)$. When we have a map to the $(p+1)$-skeleton, the obstruction to changing it to lie in the $p$-skeleton via framed bordism over the $(p+1)$-skeleton lies in $\Omega_{3}^{\fr}(X^{(p+1)},X^{(p)})$. But this lies on the $E^1_{p,q}$ page of the

Consider the Atiyah-Hirzebruch spectral sequence for $\Omega_*^{\fr}(-)$, with $E^2$ page $E^2_{p,q}= H_p(L(n,1);\Omega^{\fr}_q)$ and converging to $\Omega_{p+q}^{\fr}(L(n,1))$. By inspection of the differentials, this sequence collapses already at the $E^2$ page, so the groups $H_p(L(n,1);\Omega^{\fr}_q)$ form iterated graded quotients for a filtration of $\Omega^{\fr}_3(L(n,1))$. We note that by the Pontryagin-Thom theorem, we have $\Omega_q^{\fr}= \Z,\Z/2,\Z/2,\Z/24$ when $q=0,1,2,3$ respectively.

We analyse the groups on the $E^2$ page in turn by considering the CW decomposition of $L(n,1)$ with a single $p$-cell $e^p$ for each of $p=0,1,2,3$ and recalling that the $E^1$ page is given by
\[
E^1_{p,q}=C^{\text{cell}}_p(L(n,1);\Z)\otimes\Omega^{\fr}_{q}\cong \wt{\Omega}^{\fr}_{q}(S^0) \cong  \wt{\Omega}_{p+q}^{\fr}(e^{p}/\partial e^{p})\cong \Omega_{p+q}^{\fr}(X^{(p)},X^{(p-1)})
\]
where $X^{(p)}$ denotes the $p$-skeleton of $L(n,1)$. The groups on the $E^2$ page can be considered as a sequence of obstructions to finding a framed null-bordism of $S^3_n(K)\sqcup-L(n,1)$ over $L(n,1)$ so we analyse these obstructions in turn.

At $q=0$, a representative class for $[S_n^3(K),f] - [L(n,1),\Id]\in\Omega_3^{\fr}(L(n,1))$ is given by a cycle in $E^1_{3,0}=\Omega^{\fr}_3(X^{(3)},X^{(2)})$. The resulting class in $E^2$ vanishes if we can do surgery on the disjoint union $S_n^3(K)\sqcup- L(n,1)$ in a way that respects the degree one normal maps to $L(n,1)$ and such that the effect of surgery maps to the $2$-skeleton of $L(n,1)$. In other words we seek to compatibly stably frame the connected sum $S_n^3(K)\#(- L(n,1))$. The obstruction to doing this is the difference of $f_*[S_n^3(K)]-\Id_*[L(n,1)]\in H_3(L(n,1);\Omega_0^{\fr})$, which vanishes because the maps $f$ and $\Id$ are degree one.

Next, for $q=1$ and $q=2$ homology computations using that $n$ is odd show that both $H_2(L(n,1);\Omega^{\fr}_1) = 0$ and $H_1(L(n,1);\Omega^{\fr}_2)=0$, so there is no obstruction here. Thus we may assume we have done surgery on $S_n^3(K)\sqcup (-L(n,1))$ over $L(n,1)$ to obtain some closed, connected $Y$ together with a degree one normal map $g\colon Y\to L(n,1)$, where the map $g\colon Y\to L(n,1)$ has codomain the $0$-skeleton, and $[Y,g]=[S_3(K),f)]-[L(n,1),\Id]\in\Omega_3^{\fr}(L(n,1))$.

Finally, for $q=3$, the last $E^2$ page obstruction is given by a class in $E^1_{0,3}=\Omega_3^{\fr}(X^{(0)},\emptyset)$ which is equal to $[Y]\in \Omega_3^{\fr}\cong\Z/24$. By framed bordism invariance, this is equal to $[S_n^3(K)]-[L(n,1)]\in H_0(L(n,1);\Omega^{\fr}_3)\cong \Omega_3^{\fr}\cong\Z/24$. Similarly to the proof of Lemma~\ref{lem:normal-bordism-even-n}, the framing near a point in $S_n^3(K)$ may be modified in a small neighbourhood to force vanishing of this obstruction.

As the various representative elements of $[S_n^3(K),f] - [L(n,1),\Id]\in \Omega_3^{\fr}(L(n,1))$ can be made to vanish on the $E_2$ page, this class vanishes in $\Omega^{\fr}_3(L(n,1))$, and we obtain the required bordism. Similarly to the proof of Lemma \ref{lem:normal-bordism-even-n}, this bordism over $L(n,1)$ is seen to be a degree one normal bordism.
\end{proof}

From this point onwards we require the definition of a degree one normal map in the topological category. The necessary definitions are identical to those in Definitions~\ref{def:smoothnormal} and~\ref{defn:DONM}, except now we use the \emph{topological stable normal bundle}, which is a topological $\R^\infty$-bundle classified by a map into the classifying space $\BTOP$; see e.g.~\cite[Definition 7.12]{TheGuide}. There is a forgetful map $\BO\to \BTOP$, under which the stable normal vector bundle of a smooth manifold is sent to the stable topological normal bundle of the underlying topological manifold. The results derived so far in this section thus descend to statements in the topological category.

\begin{proposition}\label{prop:alltogether}
Let $K$ be a knot, and let $n$ be odd. There is a degree one normal bordism $F\colon W\to L(n,1)\times I$ from the map $f\colon S_n^3(K) \to L(n,1)$ of Lemma \ref{lem:deg-one-map} to $\Id \colon L(n,1)\to L(n,1)$, such that $\pi_1(W)\cong\Z/n$ and $\sigma(W)=0$. The same is true for $n$ even when $\Arf(K)=0$.
\end{proposition}

\begin{proof}
When $n$ is odd, we use Lemma \ref{lem:normal-bordism-odd-n}, and when $n$ is even we use Lemma \ref{lem:normal-bordism-even-n}, to obtain a (smooth) degree one normal bordism $W$. Now perform (smooth) 1-surgeries on the interior of $W$ to modify the fundamental group to $\Z/n$ while still retaining a degree one normal map. We abuse notation and continue to write $F\colon W\to L(n,1)\times I$ for this degree one map, so that now $\pi_1(W) \cong \Z/n$. Next take connected sums with the degree one normal map $E_8\to S^4$ (or the oppositely oriented version $-E_8\to S^4$), in order to kill the signature of $W$. Here $E_8$ denotes the $E_8$ manifold, that is, a closed, simply connected topological $4$-manifold with intersection form given by the $E_8$ matrix. The result may no longer be smooth, but now has all the desired properties.
\end{proof}

The objective is now to perform 2-surgeries on $W$, so that after the surgeries the map $W\to L(n,1)\times I$ is a homotopy equivalence. This will imply that $W$ is a $\Z[\Z/n]$-coefficient homology bordism from $S_n^3(K)$ to $L(n,1)$.
The obstruction to performing these surgeries is given by the following proposition and the obstruction group is analysed in the next section of the article.

\begin{proposition}\label{prop:alltogether-2}
Let $K\subset S^3$ be a knot satisfying $H_1(S_n^3(K);\Z[\Z/n]) = 0$ and suppose there is a degree one normal bordism $(W,F)$ as in Proposition~\ref{prop:alltogether}. Then $(W,F)$ determines a surgery obstruction in the group $L_4^s(\Z[\Z/n])$.
\end{proposition}

\begin{proof}
Since $H_2(L(n,1) \times I;\Z[\Z/n]) =0$, the surgery kernel module \cite[\textsection 5]{Wall-Ranicki:1999-1} is equal to $H_2(W;\Z[\Z/n])$ and the middle-dimensional $\Z[\Z/n]$-coefficient intersection pairing \[\lambda \colon H_2(W;\Z[\Z/n]) \times H_2(W;\Z[\Z/n]) \to \Z[\Z/n].\]
is equal to the surgery kernel pairing. We assumed that $H_1(S_n^3(K);\Z[\Z/n]) = 0$, so this pairing is nonsingular. This intersection form, together with the Wall self-intersection form $\mu$, determine a quadratic form $(H_2(W;\Z[\Z/n]),\lambda,\mu)$ over $\Z[\Z/n]$.

We claim this quadratic form is moreover \emph{simple} (with respect to some stable $\Z[\Z/n]$ basis). The manifold with boundary $(W,S_n^3(K)\sqcup-L(n,1))$ is a simple Poincar\'{e} pair \cite[Essay III, Theorem 5.13]{Kirby-Siebenmann:1977-1}, so it remains to check that the $\Z[\Z/n]$-homology equivalence on the boundary
\[
f\sqcup \id\colon S^3_n(K)\sqcup -L(n,1)\to L(n,1)\sqcup -L(n,1)
\]
has vanishing Whitehead torsion. But the identity map has vanishing Whitehead torsion, and so does $f$ by Lemma \ref{lem:whitehead}.
\end{proof}

\section{The surgery obstruction}\label{sec:homologycob}

For all $n$, the surgery obstruction group $L_4^s(\Z[\Z/n])$ has been computed and is given by a collection of signatures known as the \emph{multisignature}. In this section we describe the multisignature and then relate it to Tristram-Levine signatures of a knot in the case of interest to us.

\subsection{The multisignature for a general finite group}

We first recall some representation theory. Let $\pi$ be any finite group and let $\C\pi$ be the complex group ring. The elements of the \emph{representation ring} $R_\C\pi$  are formal additive differences of finitely generated $\C\pi$-modules. The product structure is given by tensor product; we note the formal addition also agrees with the direct sum of modules.

Given a finitely generated $\C\pi$-module $V$, we obtain the underlying complex vector space, denoted $V_\C$, by forgetting the $\pi$ action. For example, the rank $1$ free module~$V = \C\pi$ has underlying $|\pi|$-dimensional complex vector space, which has a natural basis given by the group elements of~$\pi$. We are emphasising the distinction between $V$ and $V_\C$ to avoid confusion between rank and dimension, and to increase clarity in later proofs. A $\C\pi$-module $V$ determines a $\C\pi$-module structure on the complex conjugate vector space $\overline{V_\C}$, and this is called the \emph{complex conjugate representation}. Those $\C\pi$-modules that are isomorphic to their own complex conjugate representation form a subring $R_\C^+\pi\subset R_\C\pi$ called the \emph{purely real representation ring}. Given a $\C\pi$-module $V$, the \emph{character} is $\chi_V\colon \pi\to \C$, where $\chi_V(g)$ is the trace of the endomorphism of $V$ given by $g$.

Now let $\lambda\colon V\times V\to \C\pi$ be a hermitian form on a finitely generated $\C\pi$-module $V$. The form $(V,\lambda)$ determines a $\pi$-equivariant hermitian form $(V_\C,\lambda_\C)$ over $\C$, where

%Given a finitely generated $\C\pi$-module $V$, write $V_\C$ for the underlying complex vector space (which is the same as $V$ setwise). For example, the rank 1 $\C\pi$ module $\C\pi=V$ with the $|\pi|$-dimensional complex vector space $V_\C$ with complex basis given by the elements of $\pi$, or more generally a rank $k$ free $\C\pi$ module $V$ becomes a $k|\pi|$-dimensional complex vector space $V_\C$. Now let $\lambda\colon V\times V\to \C\pi$ be a hermitian form on a finitely generated $\C\pi$-module $V$. The form $(V,\lambda)$ determines a $\pi$-equivariant hermitian form $(V_\C,\lambda_\C)$ over $\C$, where
\[
\lambda_\C\colon V_\C\times V_\C\xrightarrow{\lambda}\C\pi\xrightarrow{\text{trace}}\C
\]
and where ``trace'' denotes taking the coefficient of the neutral element $e\in\pi$. We may take $V^+$ and $V^-$, the maximal positive definite and negative definite subspaces with respect to $\lambda_\C$. These subspaces are moreover $\pi$-invariant and hence are themselves $\C\pi$-modules. We define the \emph{representation-valued multisignature}:
\[
\text{mult}(V,\lambda)= V^+-V^-\in R_\C\pi.
\]
Given a nonsingular hermitian form $(V,\lambda)$ representing an element of  $L_4^s(\Z[\pi])$, we may complexify to obtain a hermitian form over $\C\pi$, and then take the representation-valued multisignature. This determines a group homomorphism
\[
\text{mult}\colon L_4^s(\Z[\pi])\to R_\C\pi.
\]

Taking characters determines an injective ring homomorphism $R_\C\pi\to \Hom^{\text{class}}(\pi,\C)$, where the latter denotes the ring of $\Z$-linear combinations of complex-valued functions on $\pi$ that are constant on conjugacy classes of $\pi$. We call the image of $\text{mult}(V,\lambda)$ under this homomorphism the \emph{character-valued multisignature}.

\subsection{The multisignature for a finite cyclic group}
We now restrict our attention to $\pi=\Z/n$ and choose a generator $t$, obtaining an isomorphism $\C[\Z/n]\cong \C[t]/(t^{n}-1)$. Let $\chi$ represent the character of the irreducible $\C[t]/(t^n-1)$-module $\C$, where $t$ acts by $\exp(2\pi i/n)$. The characters for the irreducible $\C[\Z/n]$-modules are then $\chi^0, \chi^1, \chi^2, \dots,\chi^{n-1}$, and the ring of complex class functions for $\Z/n$ is well-known to be isomorphic to $\Z[\chi]/(1+\chi+\chi^2+\dots+\chi^{n-1})$.

\begin{proposition}\label{prop:multisignature} Let $(V,\lambda)$ be a hermitian form over $\C[\Z/n]$, with associated $\pi$-equivariant complex hermitian form $(V_\C,\lambda_\C)$. The coefficient $\alpha_k$ of the character-valued multisignature
\[
\mult(V,\lambda)=\alpha_0\cdot \chi^0+\alpha_1\cdot \chi^1+\dots+\alpha_{n-1}\cdot\chi^{n-1}\in \Z[\chi]/(1+\chi+\chi^2+\dots+\chi^{n-1})
\]
is equal to the ordinary signature of the restriction of $\lambda_\C$ to the $\exp(2\pi ik/n)$-eigenspace of the action of $t$ on $V_\C$.
\end{proposition}

\begin{proof}
Recall the positive and negative definite $\C[\Z/n]$-modules $V^+$ and $V^-$. There is then a decomposition into irreducible components $V^\pm\cong\bigoplus_{k=0}^{n-1}V^\pm_k$, where $V^\pm_k$ denotes the component that corresponds to the character $\chi^k$. The representation-valued multisignature then decomposes as
\[
\text{mult}(V,\lambda)=V^+-V^-=\sum_{k=0}^{n-1}(V^+_k-V^-_k)\in R_\C(\Z/n)
\]
Taking characters, we see that $\alpha_k=\dim_\C(V_k^+)-\dim_\C(V_k^-)$. In other words, $\alpha_k$ is the ordinary signature of the restriction of $\lambda_\C$ to $V_k=V_k^+\oplus V_k^-$. Note that, viewing $V_k$ as a complex vector subspace $V_k\subset V_\C$, it is the $\exp(2\pi i k/n)$-eigenspace of the action of the generator $t\in\Z/n$ on $V_\C$.
\end{proof}

\begin{example}\label{ex:plugitin} Suppose $(V,\lambda)$ as in Proposition \ref{prop:multisignature} and $A(t)$ is a hermitian matrix over $\C[\Z/n]\cong\C[t]/(t^n-1)$ representing $(V,\lambda)$. Then the signature of $\lambda_\C$ restricted to $V_k$ is given by the signature of the complex hermitian matrix $A(\exp(2\pi ik/n))$.
\end{example}

\begin{example}\label{ex:backtoreality}
Let $W$ be a compact oriented $4$-manifold with a homomorphism $\phi\colon\pi_1(W)\to\Z/n$. This determines an $n$-fold cyclic cover $\wt{W}\to W$ with covering transformation $\tau\colon \wt{W}\to\wt{W}$. The complex vector space $V_\C:=H_2(\wt{W};\C)$ has an action of $\Z/n$ generated by $\tau_*=t$, and the middle-dimensional intersection pairing $\lambda_\C$ of $\wt{W}$ is equivariant with respect to this. Thus $(W,\phi)$ determines a multisignature via $(V_\C,\lambda_\C)$. The coefficients of this multisignature may be computed using Proposition \ref{prop:multisignature} and Example \ref{ex:plugitin}.
\end{example}

For odd $n$, the next result is due to Wall \cite[Theorem 13A.4(ii)]{Wall-Ranicki:1999-1} and independently to Bak \cite{MR507109}. For $n$ a power of 2, it was explicitly derived in \cite[Corollary 3.3.3]{MR432737} and the techniques for the general even $n$ case were outlined. This general even $n$ case is stated in \cite[p.~3]{MR1747537} and implicitly calculated there.

\begin{theorem}[{\cite[\textsection 10, \textsection 12]{MR1747537}}] \label{thm:multisignature}For $n$ any positive integer, the group homomorphism
\[
\mult \colon L_4^s(\Z[\Z/n])\to \Z[\chi]/(1+\chi+\chi^2+\dots+\chi^{n-1})
\]
determined by the character-valued multisignature factors through the purely real representations $R^+_\C(\Z/n)\subset R_\C(\Z/n)$. There is moreover an isomorphism
  \[
  L_4^s(\Z[\Z/n]) \cong \left\{\begin{array}{ll} 4\Z^{(n-1)/2}\oplus8\Z&\text{$n$ odd,}\\  4\Z^{(n-2)/2}\oplus8\Z\oplus 8\Z&\text{$n$ even,}\end{array}\right.
  \]
given by the function
  \[
  (V,\lambda)\mapsto \left\{\begin{array}{ll} (\alpha_1,\alpha_2\dots,\alpha_{(n-1)/2},\alpha_0)&\text{$n$ odd,}\\  (\alpha_1,\alpha_2\dots,\alpha_{(n-2)/2}, \alpha_{n/2},\alpha_0)&\text{$n$ even.}\end{array}\right.
  \]
where $\alpha_k$ is the coefficient of $\chi^k$ in $\mult(V,\lambda)$.
\end{theorem}

\begin{remark}
The reader may be wondering why only half the character coefficients of $\mult(V,\lambda)$ appear in the above isomorphism. It is a consequence of the fact that the character-valued multisignature factors through $\R_\C^+(\Z/n)$, that the complex irreducible representations will appear as conjugate pairs in $\mult(V,\lambda)$:
\[
\arraycolsep=4pt\def\arraystretch{2}
\text{mult}(V,\lambda)=\left\{\begin{array}{ll} \alpha_0\cdot\chi^0 + \sum_{k=0}^{(n-1)/2}\,\alpha_k(\chi^k+\chi^{-k})&\text{$n$ odd,}\\ \alpha_0\cdot\chi^0 + \alpha_{n/2}\cdot\chi^{n/2}+ \sum_{k=0}^{(n-2)/2}\,\alpha_k(\chi^k+\chi^{-k})&\text{$n$ even.}\end{array}\right.
\]
In other words, for all $k$ we have that $\alpha_k=\alpha_{-k}$.
\end{remark}

\subsection{The multisignature as Tristram-Levine signatures}

For this section, let $K$ be a knot and let $W$ be a bordism from $S_n^3(K)$ to $L(n,1)$ over the group $\Z/n$. Such a bordism in particular has a multisignature as in Example \ref{ex:backtoreality}. Now we relate the Tristram-Levine signatures of $K$ to the multisignature of $W$.

%We plan to take the $n$-fold cyclic cover $\wt{W}\to W$ and ``cap off'' the boundary component $S^3\to L(n,1)$ using a certain covering $\wt{X}\to X$ that branches over an embedded $2$-sphere. We will thus obtain an $n$-fold cyclic cover $\wt{Z}\to Z$, branched over a $2$-sphere, and with boundary the $n$-fold cyclic cover $\wt{S_n^3(K)}\to S_n^3(K)$. For such a covering $\wt{Z}\to Z$, the relationship between signatures of $\wt{Z}$ and Levine-Tristram signatures of $K$ was studied by Casson and Gordon \fotnote{CITECITE}.

\begin{lemma}\label{lemma:multi-signatures}
Let $K$ be a knot and let $W$ be a cobordism from $S_n^3(K)$ to $L(n,1)$ over the group $\Z/n$. Assume that $\sigma(W)=0$. When $0<k<n$, the coefficient $\alpha_k\in\Z$ of the multisignature of $W$ at $\chi^k$ coincides with the Tristram-Levine signature of $K$ at $\xi_n^k$, where $\xi_n=\exp(2\pi i /n)$.
\end{lemma}

\begin{proof}
Noting that by hypothesis $\sigma(W)=0$, we have  by Proposition~\ref{prop:TlsigsandCGsigs} that
\[\sigma_{\xi_n^k}(K)= \sigma_k(\widetilde{W})- \sigma(W) = \sigma_k(\widetilde{W}).\]
Proposition~\ref{prop:multisignature} tells us that $\sigma_k(\widetilde{W})$ is the coefficient of the character $\chi^k$ in the multisignature as claimed.
\end{proof}

We summarise the results in the previous three sections for our purposes.

\begin{prop}\label{prop:summary}Let $K$ be a knot and let $f\colon S_n^3(K)\to L(n,1)$ be a degree one normal map. Suppose $W$ is a degree one normal bordism from $(S_n^3(K),f)$ to $(L(n,1),\Id)$ with $\sigma(W)=0$. Then the associated surgery obstruction
\[
(H_2(W;\Z[\Z/n]),\lambda,\mu)\in L_4^s(\Z[\Z/n])
\]
is trivial if and only
\[
\sigma_{\xi_n^k}(K)=0\text{ for every }\left\{\begin{array}{ll}k=1,\dots, (n-1)/2 &\text{$n$ odd,}\\ k=1,\dots, n/2& \text{$n$ even,}\end{array}\right.
\]
where $\sigma_{\xi_n^k}(K)$ denotes the Tristram-Levine signature of $K$ at $\xi_n^k$, for $\xi_n=\exp(2\pi i /n)$.
\end{prop}
\begin{proof}
By Theorem \ref{thm:multisignature}, a class in $L_4^s(\Z[\Z/n])$ vanishes if and only if there is vanishing of the associated multisignature coefficients:
\[
\begin{array}{ll} (\alpha_1,\alpha_2\dots,\alpha_{(n-1)/2},\alpha_0)&\text{for $n$ odd,}\\  (\alpha_1,\alpha_2\dots,\alpha_{(n-2)/2},\alpha_{n/2},\alpha_0)&\text{for $n$ even.}\end{array}
  \]

The coefficient $\alpha_0$ corresponds to the trivial character and thus corresponds to the eigenspace of the $t$ action on $(H_2(W;\Z[\Z/n])$ with eigenvalue 1. The signature of the $\Z[\Z/n]$-coefficient intersection form of $W$ restricted to this component is then just the ordinary signature of $W$, whose vanishing is in the hypotheses.

For $k\neq 0$, the coefficient $\alpha_k$ is equal to $\sigma_{\xi_n^k}(K)$, by Lemma \ref{lemma:multi-signatures}.
\end{proof}

\begin{proposition}\label{prop:homology-cobordism}
For odd $n$, suppose that $H_1(S_n^3(K);\Z[\Z/n])=\{0\}$ and $\sigma_{\xi_n^k}(K) = 0$ for all $0<k<n$, and $\xi_n$ a primitive $n$th root of unity. Then $S_n^3(K)$ is homology cobordant to $L(n,1)$ via a cobordism $V$ homotopy equivalent to $L(n,1) \times I$, via a homotopy equivalence restricting to the identity map on $L(n,1)$ and the standard degree one collapse map on $S^3_n(K)$. The same is true for even $n$ when $\Arf(K)=0$.
\end{proposition}

In fact, Propositions~\ref{prop:signaturesobstruct} and~\ref{prop:trivialh1} show that this is an `if and only if'.

\begin{proof}
Lemma \ref{lem:deg-one-map} shows how to construct the degree one normal collapse map $f\colon S_n^3(K)\to L(n,1)$. In Proposition \ref{prop:alltogether}, with the assumption that $\Arf(K)=0$ when $n$ even, we further constructed a degree one normal bordism $W$ over $L(n,1)$ from $(S_n^3(K),f)$ to $(L(n,1),\Id)$, such that $\pi_1(W)\cong \Z/n$, and $\sigma(W)=0$.

By the discussion at the end of Section \ref{sec:surgeryproblem}, using the fact that $H_1(S_n^3(K);\Z[\Z/n])=\{0\}$, the obstruction to performing further surgeries on $W$ to improve it to be homotopy equivalent to $L(n,1)$ lies in the group $L_4^s(\Z[\Z/n])$.

The vanishing of the surgery obstruction from Proposition \ref{prop:summary} and the fact that $\Z/n$ is a good group implies that we can perform surgery on the interior of $W$ to obtain a cobordism $V$ that is homotopy equivalent to $L(n,1)$~\cite[Chapter~11]{FQ}. In particular, $V$ is a $\Z[\Z/n]$-homology cobordism from $S_n^3(K)$ to $L(n,1)$.
\end{proof}

\section{The Arf and \texorpdfstring{$\tau$}{tau} invariants }\label{sec:fake-trace}

\subsection{The \texorpdfstring{$\tau$}{tau} invariant}
Recall that a locally flat sphere in a topological $4$-manifold $M$ is said to be \emph{generically immersed} if all its self-intersections are transverse double points. A locally flat union of discs in $M$ is said to be \emph{generically immersed} if its self-intersections are transverse double points and if the boundaries are mutually disjoint and embedded.

We refer the reader to~\cite[Chapter~1]{FQ} for standard notions such as the intersection and self-intersection numbers $\lambda$ and $\mu$ for generically immersed spheres in an ambient  topological $4$-manifold. We will use the following lemma.

\begin{lemma}[{\cite[p.~22]{FQ}}]\label{lem:lambda-mu-e}
For a generically immersed sphere~$S$ in a topological $4$-manifold, we have that
\[
\lambda(S,S)=\mu(S)+\overline{\mu(S)}+e(\nu S)
\]
where the last term denotes the euler number of the normal bundle $\nu S$ of $S$.
\end{lemma}

We also use the fact that a locally flat submanifold of a topological $4$-manifold has a linear normal bundle, as well as notions of topological transversality, immersions, and so on, for which we refer the reader to~\cite{FQ}.

The key tool in this section is the $\tau$ invariant of a generically immersed sphere with vanishing self-intersection number in a topological $4$-manifold. This was originally defined by~\cite{matsumoto78} and~\cite{freedman-kirby}, and then significantly generalised by Schneiderman and Teichner~\cite{schneiderman-teichner:2001}. We define $\tau$ in our setting, referring the reader to~\cite{schneiderman-teichner:2001} for a more general definition.

\begin{definition}
A generically immersed sphere $S$ in a topological $4$-manifold $M$ is said to be \emph{$s$-characteristic}, or \emph{spherically characteristic}, if $S\cdot R\equiv R\cdot R\mod{2}$ for every immersed sphere $R$ in $M$.
\end{definition}

\begin{proposition}[{\cite[Theorem~1, Remark~5]{schneiderman-teichner:2001}}]
Let $M$ be a simply connected topological $4$-manifold and suppose that $S$ is a generically immersed $s$-characteristic $2$-sphere in $M$ with $\mu(S)=0$. Consider the quantity
\[
\tau(S,\{W_i\}):=\sum_i S\cdot \mathring{W_i} \mod{2},
\]
 where $\{W_i\}$ is a set of framed, generically immersed Whitney discs in~$M$, pairing the self-intersection points of $S$, and such that each $W_i$ intersects $S$ transversely in double points in the interior of $\{\mathring{W_i}\}$.  Such a family of Whitney discs exists since $\mu(S)=0$.

The value of $\tau(S,\{W_i\})\in \Z/2$ does not depend on choices of pairing of double points, Whitney arcs, or Whitney discs.
\end{proposition}

Consequently we write $\tau(S) \in \Z/2$, dropping the discs $\{W_i\}$ from the notation.
When we judge that the manifold $M$ is not immediately clear from the context, we use the notation $\tau_M(S)$ instead of $\tau(S)$.

Indeed, $\tau(S)$ is an invariant of the homotopy class $[S]\in \pi_2(M)$, as follows. Each class in~$\pi_2(M)$ can be represented by a generic immersion of a sphere $S$ in $M$ by~\cite[Immersion lemma, p.~13]{FQ}.  Perform local cusp moves to ensure that ~$\mu(S)=0$ and then compute~$\tau(S)$. Representatives of a given homotopy class with vanishing self-intersection number are regularly homotopic~\cite[Proposition~1.7]{FQ}, where by definition, a regular homotopy in the topological category is a concatenation of finger moves and (embedded) Whitney moves. (If $M$ admits a smooth structure, a generic regular homotopy can be decomposed into a sequence of such moves, and every regular homotopy can be perturbed to such a homotopy~\cite[Section~III.3]{GoGu}.) It is clear that either such move preserves $\tau(S)$. Consequently, the invariant $\tau$ is well-defined on homotopy classes, and we sometimes use the notation $\tau(x)$ for an $s$-characteristic class $x\in\pi_2(M)$ for a simply connected topological $4$-manifold $M$.

In particular, if such an $x\in \pi_2(M)$ is represented by a locally flat, embedded sphere, then $\tau(x)=0$, and thus $\tau$ gives an obstruction for a homotopy class to contain a locally flat embedding. Observe that $\tau$ does not see the orientation of a sphere, that is, $\tau(x)=\tau(-x)$ for an $s$-characteristic class $x\in\pi_2(M)$ in a simply connected topological $4$-manifold~$M$.

%\begin{remark}
%Given two immersed spheres $S$ and $S'$ in a simply connected $4$-manifold $M$ with $\mu(S)=\mu(S')=\lambda(S,S')=0$, let $S+S'$ denote a sphere obtained by tubing $S$ and $S'$ together in $M$. Then $\mu(S+S')=0$. If $S+S'$ is $s$-characteristic, the quantity $\tau(S+S')$ is well-defined since different choices of tubes produce homotopic spheres.
%\end{remark}

Next we recall a well-known formulation of the Arf invariant of a knot.

\begin{proposition}[{\cite{matsumoto78, freedman-kirby}\cite[Lemma~10]{cst-twisted-whitney}}]\label{prop:arf-definition}
Let $K$ be a knot in $S^3$ bounding a generically immersed disc $\Delta$ in $D^4$ with $\mu(\Delta)=0$. Since $\mu(\Delta)=0$, there exists a collection $\{W_i\}$ of framed, generically immersed Whitney discs pairing up the self-intersections of $\Delta$ and intersecting $\Delta$ in transverse double points in the interior of $\{\mathring{W_i}\}$. Then
\[
\Arf(K)=\sum_i \Delta \cdot \mathring{W_i} \mod{2}.
\]
\end{proposition}

For a sketch of a proof, see~\cite[Lemma~10]{cst-twisted-whitney}. We give a brief outline here for the convenience of the reader. The Arf invariant of a knot $K$ is equal to $\sum_i \lk(a_i,a_i^+)\lk(b_i,b_i^+)\mod{2}$ where $\{a_i,b_i\}$ is a symplectic basis for the first homology of some Seifert surface $F$ of $K$, represented by simple closed curves with $|a_i \pitchfork b_j| = \delta_{ij}$~\cite[Section~10, p.~105]{Lickorish:1993-1}. Given such a surface $F$ and curves $\{a_i,b_i\}$, construct an immersed disc bounded by $K$ by pushing the interior of $F$ into $D^4$ and surgering along correctly framed immersed discs bounded by the curves~$\{a_i\}$ (correct framing can be arranged by boundary twisting).  Construct Whitney discs for the  self-intersections using the immersed discs bounded by the $\{b_i\}$. Then all intersections are created in pairs, except for the $\sum \lk(a_i,a_i^+)\lk(b_i,b_i^+)$ intersections created when adjusting the framings of the discs bounded by $\{a_i,b_i\}$. This equals the Arf invariant of $K$ via the Seifert form definition.

That the count in Proposition~\ref{prop:arf-definition} is well-defined follows from glueing together two generically immersed discs bounded by $K$,  producing an $s$-characteristic generically immersed sphere in $S^4$ with vanishing $\tau$ (since $\pi_2(S^4)=0$).

The next lemma shows, in particular, that knots with Arf invariant $1$ are not $n$-shake slice.  This was shown by Robertello~\cite{robertello} using a different proof.

\begin{proposition}\label{prop:km-arf}\label{prop:arf-obstruction-2}
Let $n$ be an integer. Let $S$ be an immersed sphere  representing a generator of $\pi_2(X_n(K))\cong \Z$ for a knot $K$ with $\mu(S)=0$. Then $\tau(S) =\Arf(K)$.
\end{proposition}

\begin{proof}
Given any null homotopy of $K$ in $D^4$, the union with the core of the $2$-handle of $X_n(K)$ is a sphere $S$ generating $\pi_2(X_n(K))$ but it might not have vanishing self-intersection number. However,  $n=\lambda(S,S)=2\mu(S) + e(\nu S)$, since $X_n(K)$ is simply connected (see Lemma~\ref{lem:lambda-mu-e}). Adding a local cusp to $S$ changes $\mu(S)$ by $\pm 1$ and the euler number $e(\nu S)$ by $\mp 2$, but does not change the homotopy class of $S$. Add local cusps as necessary to produce a sphere $S_K$, also generating $\pi_2(X_n(K))$, with $\mu(S_K)=0$ and $e(S_K)=\lambda(S_K,S_K)=n$.  Perform the local cusps inside $D^4$, so $S_K$ still intersects the 2-handle of $X_n(K)$ only in the core.

Next, we claim that $S_K$ is $s$-characteristic. To see this, let $R$ be an immersed sphere in $X_n(K)$. Then since $S_K$ generates $\pi_2(X_n(K)) \cong \Z$ there is some $a \in \Z$ such that~$R$ has homotopy class $a[S_K]$. Then $S_K\cdot R = an$ and $R\cdot R = a^2 n$, which have the same parity. So $S_K$ is indeed $s$-characteristic.

By definition, $\tau(S_K)$ is computed as
\[
\tau(S_K) = \sum_{i} S_K \cdot \mathring{W_i} \mod 2,
\]
%\fotnote{We should be referencing Matsumoto for this invariant, I think. 1978 milgram (editor) algebraic and geometric topology (Proceedings of Symposia in Pure Mathematics Vol XXXII Part 2).djvu. Matsumoto also cites Kaplan for the Arf K description.  Kaplan: Constructing framed 4-manifolds with given almost framed boundaries. But Kaplan's geometric definition is a bit different and entirely 3-dimensional, using triple points of an immersed disc. Well I still haven't found  proof that Arf K is given by the Whitney disc description.  Schneiderman in half gropes paper just has it as the definition of Arf of a knot. I wouldn't be surprised if they never cleanly wrote it down.   I know that a proof can be extracted from COP, since there we show that the number of twisted discs in a twisted order 2 tower mod 2 equals Levine's formula for the Arf invariant. But boundary twisting the twisted discs gives a tower like in the definition of Arf K we use here, i.e.\ framed Whitney discs intersection the order 0 disc. On the other hand boundary twisting to pair up intersections is used in Lemmas 2.3 and 2.4 of COP to go the other way, allowing twisted Whitney discs but must be disjoint from order zero.  Perhaps we should spell this all out, if there isn't a good reference anywhere? With apologies for proving a folklore theorem / morally  due to Matsumoto and Freedman/Kirby, except none of them proved it.  Since we can't find a clean statement in the literature we give one, that sort of thing.}
where~$\{W_i\}$ is some collection of framed, generically immersed Whitney discs pairing up all the self-intersections of $S_K$ and intersecting $S_K$ in transverse double points in the interior of the~$\{W_i\}$. Since the self-intersections of $S_K$ lie in the $0$-handle $D^4$ of $X_n(K)$, we may and shall assume that the discs $\{W_i\}$ also lie in $D^4$.  In that case, the value of $\tau(S_K)$ is computed purely in $D^4$ and equals $\Arf(K)$ by Proposition~\ref{prop:arf-definition}.
\end{proof}

Recall that by Freedman's classification~\cite{F} of closed, simply connected $4$-manifolds, there exists a closed topological 4-manifold homotopy equivalent but not homeomorphic to~$\CP^2$, known as the \emph{Chern manifold}, and denoted $*\CP^2$.
To build $*\CP^2$, attach a $2$-handle to $D^4$ along a $+1$-framed knot~$J$ in $S^3 = \partial D^4$ with $\Arf(J)=1$, such as the figure eight knot, and cap off the boundary, which is a homology 3-sphere, with a contractible $4$-manifold $C$, which can be found by~\cite[Theorem~1.4$'$]{F}, \cite[9.3C]{FQ}.

The next lemma shows how to use the $\tau$ invariant to distinguish between the manifolds $\CP^2\#\overline{\CP^2}$ and $*\CP^2\#\overline{\CP^2}$.

\begin{lemma}\label{lem:tau_cp2-or-star}
Let $x$, $*x$, and $\overline{x}$ denote generators of, respectively, $\pi_2(\CP^2)\cong\pi_2(*\CP^2)\cong \pi_2(\overline{\CP^2})\cong \Z$. Then $x+\overline{x}$ and $*x+\overline{x}$ are $s$-characteristic classes in $\pi_2(\CP^2\#\overline{\CP^2})$ and $\pi_2(*\CP^2\#\overline{\CP^2})$ respectively. Moreover, $\tau(x+\overline{x})=0$ while $\tau(*x+\overline{x})=1$.
\end{lemma}

\begin{proof}
Let $S$ denote the sphere $\CP^1\subset\CP^2\subset \CP^2\#\overline{\CP^2}$ and let $\overline{S}$ denote the sphere $\overline{\CP^1}\subset\overline{\CP^2}\subset \CP^2\#\overline{\CP^2}$. Then without loss of generality, $x=[S]$ and $\overline{x}=[\overline{S}]$.

First we consider $\CP^2\#\overline{\CP^2}$. The pair $\{x,\overline{x}\}$ generates $\pi_2(\CP^2\#\overline{\CP^2})$. Then
\[
(x+\overline{x})\cdot (ax+b\overline{x}) \equiv a+b \equiv a^2 +b^2 \equiv (ax+b\overline{x})\cdot (ax+b\overline{x})\mod{2}.
\]
Thus, $x+\overline{x}$ is $s$-characteristic in $\pi_2(\CP^2\#\overline{\CP^2})$. Then $\tau(x+\overline{x})=0$ since the class $x+\overline{x}$ can be represented by an embedded sphere in $\CP^2\#\overline{\CP^2}$ produced by tubing $S$ and $\overline{S}$ together.

As described above, $*\CP^2$ is built by capping off the surgery trace $X_1(J)$ of some knot $J$ with $\Arf(J)=1$ by a contractible $4$-manifold.
Inclusion induces an isomorphism~$\pi_2(X_1(J))\cong\pi_2(*\CP^2)$ and as before, the generator is $s$-characteristic. By Proposition~\ref{prop:arf-obstruction-2}, the image of $*x$ in $X_1(J)$ can be represented by a sphere $*S\subset X_1(J)$ with trivial self-intersection number and self-intersections paired by a family of Whitney discs $\{W_i\}$ so that $\sum *S\cdot \mathring{W}_i \mod 2 =\Arf(J)=1$. Since $X_1(J)\subset *\CP^2$, both $*S$ and the same Whitney discs are contained within $*\CP^2$.

We know that the pair $\{*x, \overline{x}\}$ generates $\pi_2(*\CP^2\#\overline{\CP^2})$ and the corresponding intersection form is $[+1]\oplus [-1]$. The same calculation as above shows that $*x+\overline{x}$ is $s$-characteristic in $\pi_2(*\CP^2\#\overline{\CP^2})$.

Now to compute $\tau(*x+\overline{x})$, we tube together $*S$ and $\overline{S}$ in $*\CP^2\#\overline{\CP^2}$. Call the result $*S+ \overline{S}$. Since $\overline{S}$ is embedded and $S$ and $*S$ are disjoint, the discs $\{W_i\}$ from above form a complete set of Whitney discs for the self-intersections of $*S+\overline{S}$. Moreover, the discs~$\{W_i\}$ do not intersect~$\overline{S}$. It follows that $\tau(*x+\overline{x})=\Arf(J) = 1$.
\end{proof}

\subsection{Equating the Arf and Kirby-Siebenmann invariants}

Now we bring in the ingredients from the previous two sections.  Our aim is to apply the $\tau$ invariant to prove that $\Arf(K)$ computes the Kirby-Siebenmann invariant of a certain $4$-manifold.

Recall that we built a homology cobordism $V$ between $S^3_n(K)$ and the lens space $L(n,1)$ in Proposition~\ref{prop:homology-cobordism}. The following technical lemma will be used below to show that the union of $V$ and $D_n$, the $D^2$-bundle over $S^2$ with euler number $n$ is homeomorphic to $\mathbb{CP}^2 \# \ol{\mathbb{CP}^2}$ when $n$ is odd and $\Arf(K)=0$.

 \begin{lemma}\label{lem:U-homeo}
Let $n$ be an odd integer. Let $D_n$ denote the $D^2$-bundle over $S^2$ with euler number $n$. Suppose that $S_n^3(K)$ is homology cobordant to $L(n,1)$ via a cobordism $V$ which is homotopy equivalent to $L(n,1)\times I$ via a homotopy equivalence $h$ restricting to the identity on $L(n,1)$ and the map $f$ on $S_n^3(K)$ from Lemma~\ref{lem:deg-one-map}. Let $Z$ denote the space $-X_n(K)\cup_{S_n^3(K)} V\cup_{L(n,1)} D_n$.

If $\Arf(K)=0$ then $Z$ is homeomorphic to $\mathbb{CP}^2\#\ol{\mathbb{CP}^2}$. If $\Arf(K)=1$ then $Z$ is homeomorphic to $*\mathbb{CP}^2\#\ol{\mathbb{CP}^2}$.
 \end{lemma}

 \begin{proof}
Recall from Lemma~\ref{lem:deg-one-map} that the map $f\colon S_n^3(K)\to L(n,1)$ extends to a homotopy equivalence $\overline{f}\colon X_n(K)\to D_n$ (this is merely the collapse map).
Observe that $D_n\cup_{L(n,1)} L(n,1)\times I \cup_{L(n,1)} D_n$ is an alternate decomposition of  the double of $D_n$, which is homeomorphic to  $S^2\widetilde{\times} S^2$ since $n$ is odd.

Now we define a function $G\colon Z\to S^2\widetilde{\times} S^2$.
\begin{equation*}
\begin{tikzcd}[column sep=2pt]
Z\arrow[d, swap, "G"]	&:=	&-X_n(K)\arrow[d, swap, "\overline{f}"]	&\cup	&V\arrow[d, swap, "h"]	 &\cup	&D_n\arrow[d, swap, "\id"]\\
S^2\widetilde{\times}S^2	&:=	&-D_n	&\cup	&L(n,1)\times I	&\cup	&D_n
\end{tikzcd}
\end{equation*}
We now show that $G$ is a homotopy equivalence.
Note that $Z$ is constructed from $V\cup D_n$ by adding a $2$-handle and then a $4$-handle, namely, the handles constituting $X_n(K)$ turned upside down. It is then easy to compute using the Seifert-van Kampen theorem and the Mayer-Vietoris sequence that $\pi_1(Z)=1$, $H_2(Z)\cong \Z\oplus \Z$, $H_4(Z)=\Z$. All other reduced homology groups vanish. Thus, $Z$ has the same homology groups as $S^2\widetilde{\times} S^2$. Note further that $H_2(Z)$ is generated by $[S]$ and $[\Delta]$, where $S$ is the base sphere of $D_n$ and $\Delta$ is represented by the union of the cocore of the $2$-handle of $X_n(K)$ and some null-homology for the meridian $\mu_K$ of $K$ in the surgery diagram for $S_n^3(K)$ in $V\cup D_n$. We will choose a specific $\Delta$ presently.

We know that $H_2(S^2\widetilde{\times} S^2)\cong \Z\oplus \Z$. Consider again the decomposition \[S^2\widetilde{\times} S^2=-D_n\cup_{L(n,1)} L(n,1)\times I \cup_{L(n,1)} D_n.\] Let $C$ denote the cocore of the $2$-handle of $D_n$. The boundary is the meridian $\mu_{U}$ of the unknot $U$ in the surgery diagram for $L(n,1)$ as the boundary of $D_n$.
We see that $H_2(S^2\widetilde{\times} S^2)$ is generated by $S$, the base sphere of $D_n$, and $\Delta_0:= -C\cup \mu_{U}\times I\cup C$. Note that $\Delta_0\cdot \Delta_0=0$. Now define $\Delta:=G^{-1}(\Delta_0)$. By construction, $f^{-1}(\mu_{U})=\mu_K$. Moreover, $\overline{f}^{-1}(C)$ is a cocore of the $2$-handle of $X_n(K)$ by Lemma~\ref{lem:deg-one-map}. Thus $G^{-1}(\mu_{U}\times I\cup C)$ is a null-homology for $\mu_K\subset S_n^3(K)$ in $V\cup D_n$ and $\Delta$ has the form promised in the previous paragraph. Consequently, $H_2(Z)\cong \Z\oplus \Z$ is generated by $\{S,\Delta\}$ and $H_2(S^2\widetilde{\times} S^2)\cong \Z\oplus \Z$ is generated by $\{S,\Delta_0)$. By construction, $G(S)=S$ and $G(\Delta)=\Delta_0$.  The map $G$ also induces an isomorphism $H_4(Z) \to H_4(S^2\widetilde{\times} S^2)$, which follows from a calculation using naturality of the Mayer-Vietoris sequences corresponding to the decompositions of $Z$ and $S^2\widetilde{\times} S^2$ in the definition of $G$ above.   Thus we have shown that $G$ induces isomorphisms on the homology groups of $Z$ and $S^2\widetilde{\times} S^2$. Since $Z$ and $S^2\widetilde{\times} S^2$ are simply connected, this completes the proof that $G$ is a homotopy equivalence by Whitehead's theorem and the fact that any $4$-manifold is homotopy equivalent to a cell complex. Moreover, since $G$ is a homotopy equivalence, $\Delta\cdot \Delta=\Delta_0\cdot \Delta_0=0$.

From now on we work only within the space $Z$. Note that $Z$ is a closed, simply connected, topological $4$-manifold with intersection form $[+1]\oplus [-1]$. By the classification of such $4$-manifolds~\cite{F}, $Z$ is homeomorphic to either $\mathbb{CP}^2\#\overline{\mathbb{CP}^2}$ or $*\mathbb{CP}^2\#\overline{\mathbb{CP}^2}$. The remainder of the proof will use Lemma~\ref{lem:tau_cp2-or-star} to determine the homeomorphism type of $Z$. To do this we need to find elements of $\pi_2(Z)$ with self-intersection $\pm 1$.

Let $S'$ denote an immersed sphere within $-X_n(K)\subset Z$ representing a generator of $\pi_2(-X_n(K))$. We know that $S'\cdot S'=-n$ within $Z$. Here the sign has changed since $Z$ contains the oriented manifold $-X_n(K)$ rather than $X_n(K)$. The pair $\{[S],[\Delta]\}$ is a basis for $H_2(Z)$ and we calculate that  $[S']=[S]-n[\Delta]$.
 Here we have also used the facts that $S'\cdot \Delta=1$, $S\cdot \Delta=1$, and $S\cdot S=n$. Add local cusps to arrange that $\mu(S')=0$ and thus $S'\cdot S'=e(S')=-n$.

Next we seek the classes in $H_2(Z)\cong \pi_2(Z)$ with self-intersection $\pm 1$. Since $n$ is odd, it can be represented as $n=2k+1$ for some integer $k$. Straightforward algebra implies that $[S]-k[\Delta]=[S']+(k+1)[\Delta]$ is the unique class, up to sign, with self-intersection $+1$ and that $[S']+k[\Delta]=[S]-(k+1)[\Delta]$ is the unique class, up to sign, that has self-intersection $-1$. That is,
\[(a[S] + b[\Delta])^2 = \pm 1 \,\Rightarrow \, a^2n + 2ab =\pm 1 \,\Rightarrow\, b= \frac{\pm1}{2a} - \frac{an}{2}\]
and so $b \in \Z$ only if $a=\pm 1$.

Our goal is to compute the $\tau$ invariant of the sum of these classes, since this determines the homeomorphism type of $Z$ by Lemma~\ref{lem:tau_cp2-or-star}. We must first check that the sum is $s$-characteristic. This is virtually the same computation as in Lemma~\ref{lem:tau_cp2-or-star}. We compute the sum  $([S]-k[\Delta])+([S']+k[\Delta])=[S]+[S']$. Let $a[S]+b[\Delta]$ be any class in $\pi_2(Z)$. Then $([S]+[S'])\cdot (a[S]+b[\Delta])\equiv a \mod{2}$ and $(a[S]+b[\Delta])\cdot {(a[S]+b[\Delta])}\equiv a^2 \mod{2}$. Since $a\equiv a^2\mod{2}$, this shows that $[S]+[S']$ is $s$-characteristic.

Finally, we compute $\tau(([S]-k[\Delta])+([S']+k[\Delta]))=\tau([S]+[S'])$. Represent the class $[S]+[S']$ by a sphere $\Sigma$ obtained by tubing together $S$ and $S'$. Observe that $\mu(\Sigma)=0$ since $\mu(S')=\mu(S)=\lambda(S,S')=0$.

In order to compute $\tau(\Sigma)$, pair up the self-intersections of $\Sigma$ by framed, generically immersed Whitney discs with pairwise disjoint and embedded boundaries. All the self-intersections of $\Sigma$ arise from self-intersections of $S'$ since $S$ is embedded and $S$ and $S'$ are disjoint. We have Whitney discs $\{W_i\}$ for the self-intersections of $S'$  within $D^4\subset -X_n(K)$ and by Proposition~\ref{prop:arf-obstruction-2}, we know that $\Arf(K)=\sum_i S'\cdot \mathring{W_i}\mod{2}$. Then $\tau(\Sigma)=\Arf(K)$ and by Lemma~\ref{lem:tau_cp2-or-star}, $Z$ is homeomorphic to $\mathbb{CP}^2\#\overline{\mathbb{CP}^2}$ if $\Arf(K)=0$ and to $*\mathbb{CP}^2\#\overline{\mathbb{CP}^2}$ if $\Arf(K)=1$. This completes the proof.
\end{proof}

\begin{proposition}\label{proposition:identifying}
Let $n$ be odd. Let $D_n$ denote the $D^2$-bundle over $S^2$ with euler number $n$. Suppose that $S_n^3(K)$ is homology cobordant to $L(n,1)$ via a cobordism $V$ which is homotopy equivalent to $L(n,1)\times I$ via a homotopy equivalence $h$ that restricts to the identity on $L(n,1)$ and to the degree one normal map $j$ on $S_n^3(K)$ from Lemma~\ref{lem:deg-one-map}. Let $X$ be the union of $V$ and $D_n$. Then $\Arf(K) = \ks(X)$.
\end{proposition}

\begin{proof}
Consider the union $Z:= -X_n(K)\cup X$. Since $\ks(X_n(K))=0$, we have that $\ks(Z) = \ks(X)$ by additivity of the Kirby-Siebenmann invariant~\cite[Theorem~8.2]{TheGuide}. From Lemma~\ref{lem:U-homeo}, we know that if $\Arf(K)=0$, then $Z$ is homeomorphic to $\mathbb{CP}^2 \# \ol{\mathbb{CP}^2}$, which is smooth and thus $\ks(Z)=0$. If $\Arf(K)=1$, we saw that $Z$ is homeomorphic to $*\mathbb{CP}^2 \# \ol{\mathbb{CP}^2}$ and we have that $\ks(*\mathbb{CP}^2 \# \ol{\mathbb{CP}^2})=\ks(*\mathbb{CP}^2)+\ks(\ol{\mathbb{CP}^2})=1$ since the Kirby-Siebenmann invariant is additive under connected sum. This completes the proof.
\end{proof}

\section{Proof of Theorem~\ref{theorem:Zn-shake-slice-thm}}

In order to prove Theorem~\ref{theorem:Zn-shake-slice-thm}, we need the following result of Boyer.

\begin{thm}[{\cite[Theorems~0.1 and 0.7; Proposition~0.8(i)]{Boyer1}}]\label{thm:Boyer}
For $i=1,2$, let $V_i$ be a compact, simply connected, oriented, topological 4-manifold with boundary a rational homology 3-sphere.

 An orientation preserving homeomorphism  $f\colon \partial V_1\to \partial V_2$  extends to an orientation preserving homeomorphism $F \colon V_1 \to V_2$ if and only if the following two conditions hold.
\begin{enumerate}
\item There exists an isomorphism $\Lambda\colon H_2(V_1) \to H_2(V_2)$, inducing an isometry of intersection forms such that the following diagram commutes:
\[
\begin{tikzcd}
0 \arrow{r} &H_2(V_1) \arrow{r} \arrow{d}{\Lambda} & H_2(V_1, \partial V_1) \arrow{r} & H_1(\partial V_1) \arrow{r} \arrow{d}{f_*} \arrow{r} & 0\\
0 \arrow{r} &H_2(V_2) \arrow{r} &H_2(V_2, \partial V_2) \arrow{r} \arrow{u}{\Lambda^*}  & H_1(\partial V_2) \arrow{r} & 0,
\end{tikzcd}
\]
where $\Lambda^*$ indicates the $\Hom$-dual of $\Lambda$, together with the implicit use of the identifications $H_2(V_i,\partial V_i)\cong H^2(V_i)\cong\Hom(H_2(V_i),\Z)$, coming from Poincar\'{e}-Lefschetz duality, and the universal coefficient theorem respectively.

\item  Either the intersection form on $H_2(V_1)$ is even or    $\ks(V_1)= \ks(V_2)$.
\end{enumerate}
\end{thm}

In our applications, the homeomorphism $f$ we propose to extend will be the identity map. We remark that when using this theorem, the specific homeomorphism $f$, or more precisely the induced map $f_*$, is highly significant.  We illustrate this with an example.

\begin{example}
A construction of Brakes gives examples of knots $K, J \subset S^3$ with homeomorphic $n$-surgeries but with non-homeomorphic $n$-traces, as follows. By \cite[Example~3]{Brakes80}, for any two distinct integers $a$ and $b$ with $|a|, |b| > 1$, the knots $K_{a,b}:= C_{a, ab^2+1}(T_{b,b+1})$ and $K_{b,a}:= C_{b, a^2b+1}(T_{a,a+1})$ have homeomorphic $(a^2b^2-1)$-surgeries. Brakes' homeomorphism induces a map on the first homology groups of the boundary, which are isomorphic to $\Z/ (a^2b^2-1)$ that is given by $[\mu(K_{a,b})] \mapsto ab [\mu(K_{b,a})]$.  Theorem~\ref{thm:Boyer} immediately implies that this homeomorphism of the boundaries does not extend to the traces. We now  argue that \emph{no} homeomorphism of the  boundaries can extend to the traces by computing certain Tristram-Levine signatures of the knots.

It is straightforward to verify that Theorem~\ref{thm:Boyer} implies that if $f \colon S^3_n(K) \to S^3_n(J)$ is a homeomorphism which extends to a homeomorphism of the $n$-traces, then $f_*([\mu_K])= \pm [\mu_J] \in H_1(S^3_n(J))$. An argument as in Proposition~\ref{prop:TlsigsandCGsigs}  then shows that if $K$ and $J$ have homeomorphic $n$-traces then $\sigma_{\xi}(K)= \sigma_{\xi}(J)$ for every $n$\ts{th} root of unity~$\xi$. However, a straightforward computation using Litherland's formula for the signatures of a satellite knot and the well-known formula for the signatures of a torus knot~\cite{Litherland:1979-1} shows that
\[\sigma_{\xi_{224}^{35}}(K_{3,5})= -64 \neq -60= \sigma_{\xi_{224}^{35}}(K_{5,3}).
\]
and hence $K_{3,5}= C_{3, 76}(T_{5,6})$ and $K_{5,3}= C_{5, 46}(T_{3,4})$ do not have homeomorphic $224$-traces despite having  homeomorphic $224$-surgeries.
We remark to those interested in shake concordance that this example shows that even homeomorphism of the $n$-surgery is not enough to imply $n$-shake concordance for general $n \in \mathbb{Z}$.
%It is known \footnot{cite} that there exists knots $K, J\subset S^3$ with homeomorphic 0-surgeries but with non-homeomorphic 0-traces. Moreover, the 0-traces of these knots have isometric intersections forms. The explanation is that there will be no choice of isometry $\Lambda$ that is compatible with any homeomorphism $S^3_0(K)\cong S^3_0(J)$ in the sense of the commutative diagram from Theorem \ref{thm:Boyer}.\footnot{PO: be more specific?}
%\footnot{AM: The 0 case is actually a little bit different- the homology compatibility requirement we state automatically vanishes, but there's an additional obstruction that isn't there in the rational homology sphere boundary setting.  I don't think there are published examples of a 0-surgery homeo not extending, but I know Lisa has a couple coming from annulus twisting.}
\end{example}

After this brief interlude, we  prove the main theorem.

\begin{proof}[Proof of Theorem~\ref{theorem:Zn-shake-slice-thm}]
Suppose that the generator of $\pi_2(X_n(K))$ can be represented by a locally flat embedded sphere $S$ such that $\pi_1(X_n(K)\sm S)\cong \Z/n$. Then Propositions~\ref{prop:trivialh1} and~\ref{prop:signaturesobstruct} establish conditions~\eqref{item:trivialh1} and~\eqref{item:signaturesobstruct} respectively, using Lemma~\ref{lem:shake-slice-to-hom-cob}. Condition~\eqref{item:trivialarf} is established in Proposition~\ref{prop:arf-obstruction-2} for all $n$, and another proof is outlined in Remark~\ref{prop:Brown-Kervaire} for even $n$.

Now consider the converse. Using Conditions~\eqref{item:trivialh1} and~\eqref{item:signaturesobstruct}, Proposition~\ref{prop:homology-cobordism} constructs a homology cobordism $V$ between $S^3_n(K)$ and $L(n,1)$ along with a homotopy equivalence to $L(n,1)\times I$ restricting to the identity map on $L(n,1)$ and the standard degree one collapse map on $S^3_n(K)$. Let $X$ denote the union of $V$ and the disc bundle $D_n$ over $S^2$ with Euler number $n$. Note that $\pi_1(X)=\{1\}$, $\pi_2(X)=\Z$, and the intersection form is $[n]$, which presents the linking form $[1/n]$ on $S_n^3(K)$.

Suppose $n$ is even. Then Boyer's classification (Theorem \ref{thm:Boyer}) implies that $X \cong X_n(K)$, extending the identity map on the boundary.
%The image of the zero section $D_n$ in $X_n(K)$ gives rise to a locally flat embedded sphere in~$X_n(K)$ representing a generator of $\pi_2(X_n(K))$.
Suppose $n$ is odd. By Proposition~\ref{proposition:identifying}, $\Arf(K)=0$ implies that $\ks(X)=0$, which implies by Boyer's classification (Theorem \ref{thm:Boyer}) that $X \cong X_n(K)$, extending the identity map on the boundary.  In either case, the image of the zero section of $D_n$ in $X_n(K)$ gives rise to a locally flat embedded sphere in~$X_n(K)$ representing a generator of $\pi_2(X_n(K))$.
\end{proof}

\section{$\pm1$-shake sliceness}\label{sec:n=1}

Recall that when $n=\pm1$, the three conditions of Theorem~\ref{theorem:Zn-shake-slice-thm} reduce to $\Arf(K)=0$. As noted in the introduction, there is a quick proof that if $\Arf(K)=0$ then $K$ is $1$-shake slice.

\begin{example}[The `if' direction when $n=1$]\label{ex:1}
Let $K$ be a knot with $\Arf(K)=0$. Since $S^3_{ 1}(K)$ is a homology sphere, it bounds a contractible 4-manifold $C$, by~\cite[Theorem~1.4$'$]{F} (see also~\cite[9.3C]{FQ}). By removing a small ball from $C$ we obtain a simply connected homology cobordism~$V$ from $S^3_{1}(K)$ to $S^3=S^3_{1}(U)$.
Define $X:= V \cup_{S^3} X_{1}(U) \cong C \# \CP^2$, and observe that $X$ is a simply connected $4$-manifold with boundary $S^3_{\pm1}(K)$ and intersection form $[ 1 ]$. By Boyer's classification (Theorem \ref{thm:Boyer}), the manifold $X$ is homeomorphic to $X_{1}(K)$ if and only if the Kirby-Siebenmann invariant~$\ks(X)=\ks(X_{1}(K)=0$.

Since the Kirby-Siebenmann invariant is additive under connected sum~\cite[Theorem~8.2]{TheGuide} and $\CP^2$ is smooth, we have $\ks(X)=\ks(C)+\ks(\CP^2) = \ks(C)$. Moreover $C$ is contractible and hence is a topological spin manifold. By~\cite[p.~165]{FQ} and~\cite{Gonzalez-Acuna}, $\ks(C)=\mu(S^3_1(K))=\Arf(K)=0$, where $\mu(S^3_1(K))$ is the Rochlin invariant of the homology sphere~$S^3_1(K)$. Thus we have a homeomorphism $X \to X_1(K)$ and the image of $\CP^1 \subset \CP^2$ is an embedded sphere with simply connected complement representing the generator of $\pi_2(X_1(K))$.
\end{example}

%Some things that are true (in the topological category) for $n= \pm 1$ that may or may not be true in general.
%%\footnot{ANM: The final few involve $n$-shake concordance- who knows if we want to define that or not. I'm including everything I can think of right now. PO: I vote for not expanding the number of things defined.}
%\begin{enumerate}
%\item If $K$ and $J$ are both $ 1$-shake slice then $K \#J $ is $1$ shake slice.
%\item $K$ is $1$-shake slice if and only if $-K$ is $1$-shake slice.\footnot{ANM: false if one works in the pseudo- 1-shake-slice setting...}
%%\item If $K$ and $J$ are 1-shake concordant and $J$ and $L$ are 1-shake concordant, then $K$ and $L$ are 1-shake concordant.
%%\item $K$ and $J$ are 1-shake concordant if and only if $K \# -J$ is 1-shake slice. \footnot{ANM: Known to be false in smooth?}
%\end{enumerate}

\subsection{A Seifert surface approach}
Now we describe yet another proof that, if $\Arf(K)=0$, then $K$ is $1$-shake slice. This proof is Seifert surface based and has the advantage that we control the number of intersections of the resulting 2-sphere with the cocore of the 2-handle, yielding an upper bound on the $1$-shaking number of knots. A similar proof for Theorem~\ref{theorem:Zn-shake-slice-thm} when $|n|\geq 2$, including similar control on the $n$-shaking number of $\Z/n$-shake slice knots, seems possible in principle, but we have not managed to find it yet.

\begin{definition}
	A $(2k+1)$-component $n$-\emph{shaking} of a knot~$K$ is a link~$S_{2k+1, n}(K)$ obtained by taking $2k+1$ push-offs with respect to the $n$-framing, oriented so that $k+1$ push-offs are oriented the same as $K$, and $k$ push-offs are oriented in the opposite direction.
\end{definition}
The link $S_{2k+1, n}(K)$ is the (untwisted) satellite link with companion~$K$ and pattern~$S_{2k+1, n}$ as in Figure~\ref{fig:SPattern}.  It is determined up to isotopy by $K$, $n$, and $k$.
\begin{figure}[h]
	\includegraphics[width=6cm]{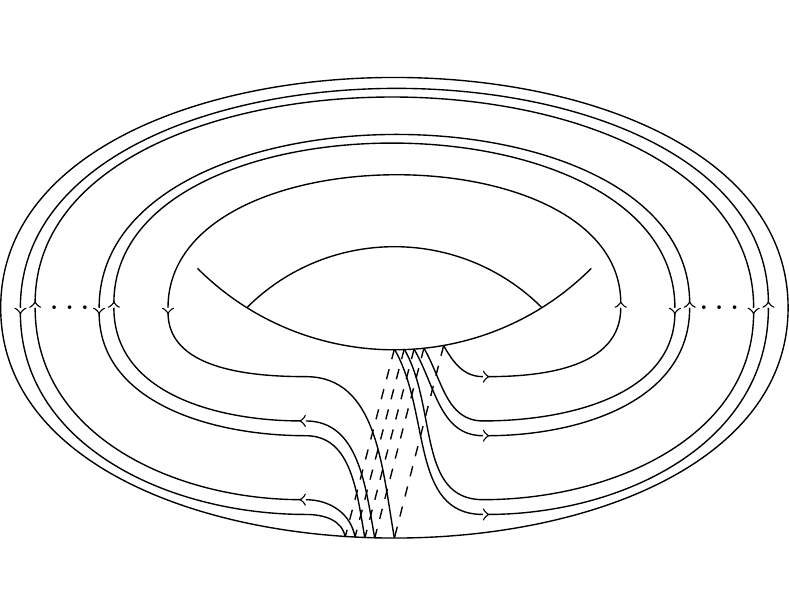}
	\caption{The pattern~$S_{3, 1}$}
	\label{fig:SPattern}
\end{figure}

Now we explain the well-known relationship between the existence of an embedded sphere in the $n$-trace~$X_n(K)$ and shakings~$S_{2k+1,n}(K)$. Recall that a link~$L$ is \emph{weakly slice} if $L$ bounds a locally flat planar surface in~$D^4$.

\begin{lemma}\label{lem:CocoreShaking} Let $K$ be an oriented knot. There exists a surface of genus~$g$ that represents the generator of~$H_2(X_n(K); \Z)$ if and only if, for some~$k \geq 0$, the shaking~$S_{2k+1,n}(K) \subset S^3$ bounds a surface of genus~$g$ in $D^4$. 
        Consequently, $K$ is $n$-shake slice if and only if $K$ admits a weakly slice $n$-shaking $S_{2k+1,n}(K)$ for some $k \in \Z$.
\end{lemma}
\begin{proof}
The if direction follows from the observation that any filling surface in $D^4$ for the shaking~$S_{2k+1,n}(K)$ can be capped of with $2k+1$ parallel copies of the two-handle to produce a surface that represents the generator of~$H_2(X_n(K); \Z)$.

For the only if direction, let $F$ be a surface of genus~$g$ that represents the generator of~$H_2(X_n(K); \Z)$.
Isotope~$F$ so that it is transverse to the cocore of the $2$--handle. Cut~$X_n(K)$ along the cocore~$C$ to obtain back~$D^4$. This punctures~$F$ to a surface~$F'$. The surface~$F'$ has the same genus as~$F$ and is bounded by the shaking~$S_{2k+1,n}(K)$, where~$2k+1$ is the geometric intersection number of $F$ with the cocore~$C$.
\end{proof}

For $1$-shakings, we provide the following explicit variant of our main result Theorem~\ref{theorem:Zn-shake-slice-thm}.
Recall that the \emph{$($topological$)$ $\Z$--slice genus} $\gZ(L)$ of a link $L$ is the smallest genus among $\Z$-slice surfaces for $L$. Here, a \emph{$\Z$--slice surface} for~$L$ is a properly, locally flatly embedded, compact, connected, and orientable surface in~$D^4$ with boundary~$L$ and infinite cyclic fundamental group of the complement.
For knots, the topological $\Z$-slice genus equals the algebraic genus~$\galg$~\cite{FellerLewark_19}, which is a quantity depending only on the $S$-equivalence class of the Seifert form of knots and that in particular satisfies $\gZ(K)=\galg(K)\leq \deg(\Delta_K)/2\leq g_3(K)$. Here $g_3$ denotes the 3-genus.

%\textit
For our purposes, it will suffice to define $\galg(K)$ to be the minimal $g$ such that there exists a Seifert surface $S$ for $K$ and a basis for $H_1(S)$ with respect to which the Seifert form is given by $\left[ \begin{array}{cc} A & B \\ B^T &C \end{array} \right]$, where $C$ is a $2g \times 2g$ matrix and $A$ is \emph{Alexander trivial}, i.e.\ $\det(tA-A^T)$ is some power of $t$.
From this definition, one can see that $\galg(K)$ gives a lower bound on the topological $\Z$-slice genus of $K$: $A$ corresponds to a subsurface $S_A$ of $S$ with boundary $\partial S_A= J_A$, where $J_A$ has trivial Alexander polynomial, and so surgering $S$ along $J_A$ gives a $\Z$-slice surface for $K$ of genus $g$. See~\cite{FellerLewark_16} for details on  $\galg$.

\begin{proposition}\label{prop:Shake1Genus}
        Let $K$ be a knot with algebraic genus~$\galg(K)$. Suppose $\galg(K)\geq h > 0$.
        Then there exists a $\Z$-slice surface with genus~$h$ for the $1$-shaking of~$K$ with $2(g_{alg}(K) - h) + 1$-components.
        If $K$ has trivial Arf invariant, then the statement also holds for $h=0$, i.e.~there exists a $\Z$-slice genus~$0$ surface for a $2g_{alg}(K) + 1$-component $1$-shaking of~$K$.
        \end{proposition}

Let $K$ be a knot satisfying the assumptions of Proposition~\ref{prop:Shake1Genus}. Thus, we obtain a $\Z$-slice surface~$S$ with boundary a $1$--shaking of~$K$ and genus~$h$. Now cap off~$S$ with parallel copies of the core of the $2$-handle of~$X_1(K)$ to obtain a closed $\Z$-slice surface of genus $h$ in $X_1(K)$, which represents a generator of~$H_2(X_1(K);\Z)$. We have shown the following corollary.

\begin{corollary}\label{cor:Shake1Genus}
        For a knot~$K$ with $\galg(K)\geq h>0$,
         the generator of~$H_2(X_1(K);\Z)$ can be represented by a locally flat genus $h$ surface whose complement is simply connected and that has geometric intersection number $2(\galg(K) -h)+1$ with the cocore of the 2-handle.
If $K$ has trivial Arf invariant, then the statement also holds for $h=0$.
\end{corollary}

Proposition~\ref{prop:Shake1Genus} and Corollary~\ref{cor:Shake1Genus} give an explicit bound on the genus and the number of points of intersection with the cocore of the 2-handle in terms of $\galg(K)$, a quantity accessible in terms of Seifert matrices, and which is bounded above by the 3-genus. In particular, this together with the next remark yield Propositions~\ref{prop:torus} and~\ref{prop:shaking-num} from the introduction.

\begin{remark}\label{rem:g4 bounds 2k+2g}
For a locally flat genus $g$ surface $F$ that represents a generator of $H_2(X_n(K);\Z)$ and that intersects the cocore of the 2-handle transversely $2k+1$ times, one has $\gs(K)\leq g+k$. Here $\gs(K)$ denotes the (topological) slice genus of a knot $K$.

Indeed, as in the proof of Lemma~\ref{lem:CocoreShaking}, we find a locally flat connected surface in $D^4$ with genus $g$ and boundary the shaking~$L=S_{2k+1,n}(K)$. Since $2k$ saddle/band moves turn $L$ into $K$, we find a genus $g+k$ locally flat surface in $D^4$ with boundary $K$. Hence $\gs(K)\leq g+k$.
\end{remark}

\begin{proposition}\label{prop:torus}
For every knot $K$ there exists a locally flat embedded torus in $X_1(K)$ that generates $H_2(X_1(K))$ and has simply connected complement. In particular, \[g^1_{\operatorname{sh}}(K) = \Arf(K) \in \{0,1\}.\]
\end{proposition}

\begin{proof}
For knots with Arf invariant $1$, setting $h=1$ in Corollary~\ref{cor:Shake1Genus} provides a locally flat torus with simply connected complement representing a generator of $H_2(X_1(K);\Z)$. For knots with Arf invariant $0$, we even find a locally flat sphere whose complement is simply connected by setting $h=0$ in Corollary~\ref{cor:Shake1Genus}.
\end{proof}

\begin{proposition}\label{prop:shaking-num}
For a knot $K$ with  $\Arf(K)=0$ we have
\[2\gs(K)+1\leq \text{$1$-shaking number of $K$}\leq2\gZ(K)+1 =2\galg(K)+1\leq 2g_3(K)+1.\]
In particular, for each integer $k \geq 0$ there exists a $1$-shake slice knot $K_k$ such that the $1$-shaking number of $K_k$ is exactly $2k+1$.
\end{proposition}

\begin{proof}
Corollary~\ref{cor:Shake1Genus} gives a sphere that has geometric intersection number $2(\galg(K))+1$ with the cocore of the 2-handle. Hence the $1$-shake slice number of knots with Arf invariant $0$ is less than or equal to $2\galg(K)+1$. We already know that $\gZ(K) = \galg(K) \leq g_3(K)$ from~\cite{FellerLewark_19}. %The first part of the statement follows since $\galg(K)\leq g_3(K)$, and,
To obtain the first inequality, note that for any 2-sphere that realises the $1$-shake slice number $2k+1$, we have  $\gs(K)\leq 0+ k = k$  by Remark~\ref{rem:g4 bounds 2k+2g}.

For the second sentence, take $K_k$ to be an Arf invariant 0 knot with $\gs(K)=g_3(K)=k$. For example, let $K_k$ be the $k$-fold connected sum of $5_2$, the twist knot with $5$ crossings. This satisfies $|\sigma(K)|=2g_3(K)=2$, and thus $\gs(K)=g_3(K)=1$ by the Murasugi-Tristram inequality relating the signature and the slice genus.
\end{proof}

Next we describe a Seifert surface (and the corresponding Seifert matrix) for~$S_{2k+1,n}(K)$. For this, one could iterate a construction of Tristram~\cite[Definition~3.1]{tristram} and extract the Seifert form from his proof of signature invariance~\cite[Theorem~3.2]{tristram}, but, for the convenience of the reader, we give an argument in Lemma~\ref{lem:SeifertShaking} below.
Although, we will only need the case $n=1$ here, we give the general statement for future reference.

For any $k$ and $n$, $S_{2k+1,n}$ is a winding number~$1$ pattern, and so we
can construct a Seifert surface~$F_{2k+1,n}$
for $S_{2k+1,n}(K)$ as the union of two pieces. The first is a surface in the solid torus identified with $\nu(K)$, and illustrated in Figure~\ref{fig:SeifertOfS}. The second is a Seifert surface for $K$ outside a tubular neighbourhood~$\nu K$.
\begin{figure}
\includegraphics{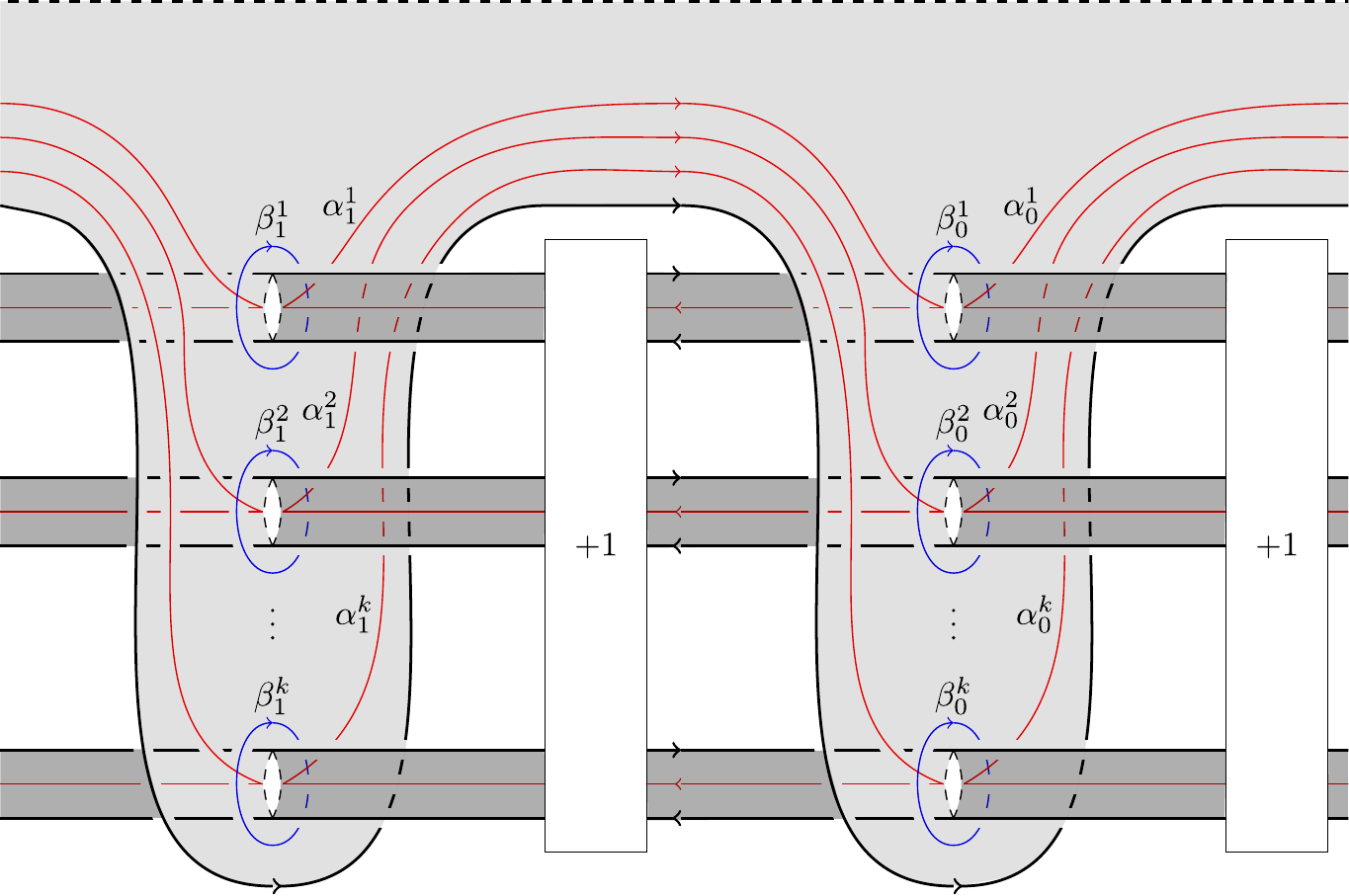}
    \caption[A portion of the Seifert surface~$F_{2k+1,n}$ for $n = 2$]{The portion of the Seifert surface~$F_{2k+1,n}$ contained in $\nu(K)$, drawn for $n = 2$ and $k=3$. Each $+1$-box denotes a positive full twist on $2k$ strands. The left and right edges of the figure are identified.  The figure also shows, in red and blue, curves on the Seifert surface forming part of a generating set for $H_1(F_{2k+1,n};\Z)$.  The dashed line on the top is glued to a Seifert surface for $K$.}
	\label{fig:SeifertOfS}
\end{figure}

Let~$Z_n \in \GL(n,\Z)$ be the permutation matrix corresponding to the cyclic permutation~$\sigma = (1 \ldots n)$, with  entries are given by $(Z_n)_{i,j}=\delta_{i, \sigma(j)}$. That is,
\[
Z_n = 	\begin{pmatrix}
		0 & 0 & \ldots & 0 & 1\\
		1 & 0 & \ldots & 0 & 0\\
		0 & 1 & \ldots & 0 & 0\\
		\vdots && \ddots & & \vdots\\
		0 & 0 & \ldots  & 1 & 0
	\end{pmatrix}.
\]
\begin{lemma}\label{lem:SeifertShaking}
	Let $V$ be a Seifert matrix for the knot~$K$.
	The Seifert surface~$F_{2k+1,n}$ for the $(2k+1)$-component $n$-shaking of $K$ depicted in Figure~\ref{fig:SeifertOfS} has Seifert form given by the matrix
	\begin{equation}\label{eq:SFofShake}
	V\oplus \bigoplus_{j=1}^k \begin{pmatrix} 0 &\id_n \\ Z_n & 0 \end{pmatrix}.
\end{equation}
\end{lemma}

\begin{proof}
	The Seifert surface~$F_{2k+1,n}$ agrees with a Seifert surface~$S$ for $K$ outside a tubular neighbourhood~$\nu K$.
	We compute the Seifert form~$(x,y) \mapsto \lk(x, y^+)$ of $F_{2k+1,n}$. The generators of~$H_1(F_{2k+1,n}; \Z)$ are given by the generators~$v_i$ of~$S$, and the additional generators~$\alpha_i^j$ and $\beta_i^j$, depicted in Figure~\ref{fig:SeifertOfS}, where $i = 0, \ldots, n-1$ and $j=1, \ldots, k$.

Denote the span $\langle v_1, \ldots, v_{2g} \rangle$ by~$C$, where $g$ is the genus of the Seifert surface of~$K$ used to construct~$F_{2k+1,n}$. Also, for each $j=1, \dots, k$ denote the span~$\langle \alpha_i^j, \beta_i^j\,|\, i = 0, \ldots, n-1\rangle$ by $D_j$. We say that two, possibly non-distinct curves $\gamma_1$ and $\gamma_2$ are \emph{orthogonal} if $\lk(\gamma_1, \gamma_2^+)= 0= \lk(\gamma_1^+, \gamma_2)$.

Examination of $F_{2k+1,n}$ gives us the following vanishing linking numbers.
\begin{enumerate}
\item For any $i,j,k$ the curve $v_k$ is orthogonal to $\alpha^j_i$  and $\beta^j_i$ since there exists a 3-ball containing each  $\alpha^j_i$ or $\beta^j_i$ which is disjoint from the surface $S$.
\item For any $i, i', j, j'$ we have that $\beta^j_i$ and $\beta^{j'}_{i'}$ are orthogonal.
\item For $i \neq i'$ and arbitrary $j, j'$ we have that $\alpha^j_i$ and $\alpha^{j'}_{i'}$ are orthogonal.
\item For arbitrary $i$, $j$ we have that $\lk( \alpha^j_i, (\alpha^j_i)^+)=0$ i.e.\ $\alpha^j_i$ is self-orthogonal.
\item For arbitrary $i, j, j'$ we have that $\alpha^j_i$ and ${\alpha^{j'}_{i}}$ are orthogonal, since ${\alpha^j_i}^+$ is isotopic to any ${\alpha^{j'}_{i}}^+$ in $S^3 \sm \alpha^j_i$ and the self-linking of $\alpha^j_i$ vanishes.
\item For $j \neq j'$ and arbitrary $i, i'$ we have
$\lk\big(\alpha^j_i, {\beta^{j'}_{i'}}^+\big) = \lk\big(\beta^j_i, {\alpha^{j'}_{i'}}^+\big) =0.$
\end{enumerate}

	In particular, we have that the first homology splits as an orthogonal sum
	\[ H_1(F_{2k+1,n};\Z) = C \oplus \bigoplus_{j=1}^k D_j, \]
	where the Seifert form on~$C$ is given by the Seifert matrix~$V$ of $K$ with respect to the basis~$v_1, \ldots, v_{2g}$.

	What remains is to compute the Seifert form on $D_j$. Fix a~$j$, and from now on suppress the index~$j$ from the notation.  Recalling that $\sigma$ is the cyclic permutation~$(1 \ldots n)$, we compute
	\[ \lk(\alpha_i, \beta_{i'}^+) = \begin{cases}
			1 & i = i',\\
			0 & \text{otherwise}
		\end{cases}
%	Moreover,, we have that
\;\;\text{  and  }\;\;
	 \lk(\beta_{i},\alpha_{i'}^+) = \begin{cases}
		1 & i' = i+1 = \sigma(i),\\
		0 & \text{otherwise.}
		\end{cases}\]
		This shows that the Seifert form for~$F_{2k+1,n}$ is indeed represented by the matrix in the statement of the lemma.
\end{proof}

\begin{proof}[Proof of Proposition~\ref{prop:Shake1Genus}]
The case~$g_{alg}(K)=0$ follows from the result of Freedman and Quinn that $\Delta_K(t)=1$ if and only if $K$ is the boundary of a $\Z$-slice disc~\cite[Theorem~7]{Freedman:1984-1},\cite[11.7B]{FQ}, since, by definition, a knot $K$ has algebraic genus~$0$ if and only if $\Delta_K(t)=1$.
So, we consider the case~$g:=g_{alg}(K)>0$.

Let $V$ be a $2m\times 2m$ Seifert matrix of $K$ that realises the algebraic genus of $K$ in the following way: the top left square block $P$ of $V$ of size $2(m-g)\times2(m-g)$ is Alexander trivial, that is $\det(tP-P^T)=t^{m-g}$. %This maybe taken as the definition of $g_{alg}$.
Additionally, arrange that the anti-symmetrisation of $V$ is a direct sum of $P-P^T$ and the standard $2g\times 2g$ symplectic form. Denoting the lower right $2g\times 2g$ square block of $V$ by $B$, the latter amounts to requiring that
\[B = \begin{pmatrix} S & A+\id_g \\ A^T & * \end{pmatrix},\]
where $A$ is a $g\times g$ matrix and $S$ is a $g\times g$ symmetric matrix.

By applying a base change that preserves the intersection form~$B-B^T$, we can and do arrange for the first $g-h$ diagonal entries of $S$ to be even. To see this, note that one easily arranges for all but at most one diagonal entries of $S$ to be even, and one may further arrange for the last entry to be even if and only if the Arf invariant is trivial. This can be checked using the formula for the Arf invariant in terms of a symplectic basis~\cite[Section~10, p.~105]{Lickorish:1993-1} and is only needed in the case that $h=0$.

Let $L$ be the $2(g-h)+1$-component $1$-shaking of $K$.
With everything set up as above, we now look at the $2(m+g-h)$-Seifert matrix $M$ of $L$ given by
$V\oplus \bigoplus_{j=1}^{g-h} \begin{pmatrix} 0 &1 \\ 1 & 0 \end{pmatrix}$ by~\eqref{eq:SFofShake}.
To establish that $L$ admits a $\Z$-slice surface of genus $h$, it suffices to find a $2(m-h)\times 2(m-h)$ Alexander trivial subblock of $M$ by~\cite[Theorem~1]{FellerLewark_16}.

For this, perform the base change corresponding to adding the basis element~$e_{2m+2l-1}$ to the basis element $e_{2(m-g)+l}$ for each~$l=1,\ldots, g-h$.
This corresponds to changing both entries $M_{2(m-g)+l,2(m+l)}$ and $M_{2(m+l),2(m-g)+l}$ from $0$ to $1$. The result is the following Seifert matrix for $L$, which we again denote by $M$.
\begin{center}
{\fontsize{8pt}{8pt}\selectfont
\renewcommand\arraystretch{1.125}
\begin{align*}
\left[\begin{array}{c|ccccc|ccccc|ccccc}
\hspace{-0.2cm}\smash{\overbrace{\hspace{0.2cm}P\hspace{0.2cm}}^{2(m-g)}}\hspace{-0.4cm}\phantom{A}&&&\hspace{-1.2cm}\smash{\overbrace{\hspace{1.2cm}C\hspace{1.2cm}}^{g}}\hspace{-1.4cm}\phantom{A}&&&&&\hspace{-1.0cm}\smash{\overbrace{\hspace{0.8cm}\phantom{A}*\phantom{A}\hspace{0.9cm}}^{g}}\hspace{-1.4cm}\phantom{A}&&&&&\hspace{-1.15cm}\smash{\overbrace{\hspace{1.2cm}0\hspace{1.15cm}}^{2(g-h)}}\hspace{-1.65cm}\phantom{A}&&\\
\hline
&           &&&&     &        &&&&&                     0&1&&&\\
&           &&&&&        &&&&&                     &&0&1&\\
C^T&        &&S&&&       &A&+&\mathrm{Id}_g&&       &&&&\ddots\\
&           &&&&&        &&&&&                     &&&&\\
&           \phantom{0}&\phantom{0}&\phantom{0}&\phantom{1}&\phantom{0}&        &&&&&                     &&&&\phantom{\ddots}\\
\hline
&           &&&&&        &&&&&                     &&&&\\
&           &&&&&        &&&&&                     &&&&\\
*&          &&A^T&&&     &&\hspace{0.1cm}*\hspace{-0.3cm}\phantom{A}&&&       			    &&\hspace{0.2cm}0\hspace{-0.4cm}\phantom{A}&&\\
&           &&&&&        &&&&&                     &&&&\\
&           &&&&&        &&&&&                     &&&&\\
\hline
&           0&&&&&       &&&&&                     0&1&&&\\
&           1&&&&&       &&&&&                     1&0&&&\\
0& &0&&&&       &&\hspace{0.1cm}0\hspace{-0.3cm}\phantom{A}&&&       			  &&0&1&\\
&           &1&&&&       &&&&&                     &&1&0&\\
&           &&\ddots&&        &&&&&                     &&&&&\ddots\\
\end{array}\right]
\end{align*}
}\end{center}
By adding multiples of the $(2m+2l)^{th}$ basis element appropriately to the first $2m-h$ basis elements, for $l=1,\dots, g-h$,
we can arrange that the first $g-h$ rows and columns of $C$, $A$, and $S$ become $0$, while no other entries in the first $2m-h$ rows and columns of $M$ are modified. This is possible since $S$ is symmetric and its first $g-h$ diagonal entries are even. We keep referring to the resulting matrix by~$M$ and conclude the proof by noting that the sub-block $M_\text{sub}$ of $M$ corresponding to the sub-basis
\[(\underbrace{e_1,\ldots,e_{2(m-g)}}_{2m-2g},\underbrace{e_{2(m-g)+1},\ldots,e_{2(m-g)+g-h}}_{g-h},\underbrace{e_{2(m-g) + g +1},\ldots,e_{2m-h}}_{g-h})\]
is Alexander trivial.

In other words, the block~$M_\text{sub}$ is the $2(m-h)\times 2(m-h)$--matrix that is obtained from $M$ by deleting the rows and columns~$2(m-g)+g-h+1$ through $2m-g$ and deleting all rows and columns with index~$2m-h+1$ or larger. This submatrix is of the form
\[M_\text{sub}=\begin{pmatrix} P & 0 &* \\ 0 & 0 & \id_{g-h}\\ *& 0&* \end{pmatrix}, \]
and so satisfies~$\det(tM_\text{sub}-M_\text{sub}^T)=\det(tP-P^T)t^{g-h}=t^{m-h}$.
\end{proof}

\bibliographystyle{alpha_labelfix}
\bibliography{research}
\end{document}